\documentclass[11pt]{amsart}
\usepackage{geometry}                % See geometry.pdf to learn the layout options. There are lots.
\usepackage{graphicx}
\usepackage{amssymb, amsmath, amsfonts, amsthm}
\usepackage{epstopdf}
\newcommand{\ZZ}{\mathbb{Z}}
\newcommand{\QQ}{\mathbb{Q}}
\newcommand{\RR}{\mathbb{R}}
\newcommand{\CC}{\mathbb{C}}
\newcommand{\PP}{\mathbb{P}}
\newcommand{\NN}{\mathbb{N}}
\newcommand{\Gr}{\mathrm{Gr}}
\newcommand{\Sym}{\mathrm{Sym}}

\newcommand{\KKK}{\boldsymbol{K}}
\newcommand{\OOO}{\boldsymbol{O}}
\newcommand{\MMM}{\boldsymbol{M}}

\newcommand{\XX}{\mathcal{X}}

\newcommand{\sss}{\mathcal{S}}

\newcommand{\MM}{\mathcal{M}}

\newcommand{\kk}{\kappa}
\newcommand{\KK}{\mathbb{K}}

\newcommand{\dashedto}{\dashrightarrow}
\newcommand{\into}{\hookrightarrow}

\newcommand{\pp}{\mathcal{P}}

\DeclareMathOperator{\Trop}{Trop}

\DeclareMathOperator{\In}{in}

\DeclareMathOperator{\Span}{Span}

\DeclareMathOperator{\Toric}{Toric}
\DeclareMathOperator{\Hilb}{Hilb}

\newcommand{\GL}{\mathrm{GL}}

\DeclareMathOperator{\Spec}{Spec}
\DeclareMathOperator{\Proj}{Proj}

\DeclareMathOperator{\Star}{Star}

\newcommand{\Hom}{\mathrm{Hom}}
\newcommand{\BT}{\mathrm{BT}}
\newcommand{\PGL}{\mathrm{PGL}}
\newcommand{\Zo}{\mathring{Z}}
\newcommand{\zo}{\mathring{z}}

\newcommand{\OO}{\mathcal{O}}

\newcommand{\DD}{\mathcal{D}}
\newcommand{\EE}{\mathcal{E}}

\newcommand{\TT}{\mathbb{T}}

\newcommand{\T}{\mathcal{T}}

\newcommand{\isomorph}{\cong}

\theoremstyle{plain}
\newtheorem{conj}{Conjecture}[section]
\newtheorem{Theorem}[conj]{Theorem}

\newtheorem{prop}[conj]{Proposition}
\newtheorem{Lemma}[conj]{Lemma}
\newtheorem{lemma}[conj]{Lemma}
\newtheorem{Corollary}[conj]{Corollary}
\newtheorem{cor}[conj]{Corollary}

\newtheorem*{TheoremGZE*}{Theorem~\ref{GenusZeroExists}}
\newtheorem*{TheoremGOE*}{Theorem~\ref{GenusOneExists}}
\newtheorem*{TheoremGOSA*}{Theorem~\ref{GenusOneSuperAbundant}}
\newtheorem*{TheoremNC}{Theorem~\ref{NecessaryConditions} -- Reduced Case}

\theoremstyle{definition}
\newtheorem{ex}[conj]{Example}

\title{Uniformizing Tropical Curves I: Genus Zero and One}
\author{David E Speyer}
\begin{document}
\maketitle
%\section{}
%\subsection{}

\begin{abstract}
In tropical geometry, given a curve in a toric variety, one defines a corresponding graph embedded in Euclidean space. We study the problem of reversing this process for curves of genus zero and one. Our methods focus on describing curves by parameterizations, not by their defining equations; we give parameterizations by rational functions in the genus zero case and by non-archimedean elliptic functions in the genus one case. For genus zero curves, those graphs which can be lifted can be characterized in a completely combinatorial manner. For genus one curves, show that certain conditions identified by Mikhalkin are sufficient and we also identify a new necessary condition.
\end{abstract}

\tableofcontents

In the past five years, a group of mathematicians, lead by Grigory Mikhalkin, have pioneered a new method for studying curves in toric varieties. According to this perspective, one considers curves defined over a field with a non-archimedean valuation. Using this valuation and an embedding of a curve $X$ into an (algebraic) torus, one constructs a graph embedded in a real vector space. This graph is known as the tropicalization of the curve. From the tropicialization of $X$, one tries to read off information about the degree and genus of the original curve $X$, and its intersections with other subvarieties of the torus. In this introduction, we will write $X$ for a curve embedded in a torus $T \isomorph (\KK^*)^n$ and we will write $\Gamma \subset \RR^n$ for the tropicalization of $X$.

In order to use these tropical methods, we need to know which graphs are tropicalizations of curves. We will refer to a graph which actually is the tropicalization of a curve as a tropical curve. There are certain basic combinatorial conditions which hold for any tropical curve. The first, the \textbf{zero tension condition}, is a description of the possible local structures of a tropical curve around a given vertex. (See the beginning of Section~\ref{Results}.) We can assign to $X$ a multiset of lattice vectors, which we will call the degree of $X$, from which we can determine the homology class represented by the closure of $X$ when this closure is taken in  a suitable toric compactification of $T$. The second combinatorial condition is that the directions of the unbounded rays of $\Gamma$ are given by the degree of $X$. (See Section~\ref{CurvesinT}.) Thirdly, we can show that, modulo some technical conditions, the genus of $X$ is greater than or equal to the first Betti number of $\Gamma$. (See Theorem~\ref{GenusBound}.) We will define a \textbf{zero-tension curve of genus $g$ and degree $\delta$} to be a graph which has first Betti number $g$ and obeys the obvious conditions to be the tropicalization of a degree $\delta$ curve.

We attack the reverse problem: Given a zero-tension curve of genus $g$ and degree $\delta$, when does it come from an actual curve of genus $g$ and degree $\delta$? The main contribution of this paper is to show that methods of non-archimedean analysis can be used to construct algebraic curves with a given tropicalization. In this paper, we will consider this question for genus zero and genus one curves. In the sequel, we will describe the corresponding results for higher genus curves, where we will need to use Mumford's uniformization results. A second achievement of this paper is to describe an obstruction to lifting zero-tension curves of genus one  -- the condition of being \textbf{well spaced} --  which is more subtle that than those mentioned above but is still combinatorial and easy to test.

I want to state clearly that there is a major difficulty in directly using these results for enumerative purposes involving curves of positive genus. If $X$ is a curve of genus $g$, than the tropicalization of $X$ is a zero-tension curve with first Betti number less than or equal to $g$. Therefore, if we want to count genus $g$ curves with obeying some conditions, we should look at all zero-tension curves of genus less than or equal to $g$, and determine which of them lift to actual genus $g$ curves obeying the condition (and in how many ways the lifting can be done). However, we have almost no results restricting the capability to lift a tropical curve  whose first Betti number is strictly less than $g$ to an actual curve of genus $g$. When studying curves in toric surfaces, one can use basic dimension counting arguments to show that there are no such contributions, but this cannot be done for curves in higher dimensional toric varieties. We expect, therefore, that the primary use of these results will not be to prove exact combinatorial formulas, but rather to provide existence results or lower bounds.

The idea of studying curves via tropical varieties was proposed by Kontsevich and pioneered by Mikhalkin \cite{Mikh1}, \cite{Mikh2}. In Mikhalkin's view, tropical varieteis should be seen as spectra of semirings, built from the semiring of the real numbers under the operations $\min$ and $+$. He also points out, however, that they can be described using valuations over a nonarchimedean field. (As we do here.) Mikhalkin has proven our main theorems, in any genus, in the case of curves in toric surfaces. A purely algebraic proof was given by Shustin and Tyomkin \cite{ST}. Since Mikhlkain's work, there has been a great deal of research extending his results to more sophisticated enumerative problems concerning curves in toric surfaces. There has been far less work on curves in higher dimensional toric varieties. The most important exception is the work of Nishinou and Siebert \cite{SiebNish}, who use log geometry to analyze the case of genus zero curves and recover essentially all of our results in that case. Finally, we should note that H. Markwig and her collaborators, especially  Gathmann, have done major work building the tropical analogue of the moduli spaces of curves and of stable maps and studying it from a combinatorial persepective, see \cite{GKM} and the works cited therein.  Among their results is reestablishing the validity of the tropical enumeration of curves in $\PP^2$ by showing it matches the Caporaso-Harris formula. The moduli space of tropical genus zero curves was previously described by Mikhalkin in \cite{Mikh3}.

This paper, which has been several years in preparation, stems from portions of my Ph.D. dissertation. I am extremely grateful to my advisor Bernd Sturmfels, for his guidance in the writing of the dissertation and afterwards. In addition, I am grateful to Paul Hacking, Thomas Markwig, Grigory Mikhalkin, Sam Payne and Jenia Tevelev for many discussions of my work and theirs, including informing me of many results which were unpublished at the time. During the writing of this paper, I have been supported by a research fellowship from the Clay Mathematics Institute. 

\section{Curves in Toric Varieties} \label{CurvesinT}

In this section, we will describe how to assign a degree to a curve given with a map to an (algebraic) torus. Throughout this paper, we will write $\Lambda$ for the lattice of one parameter subgroups of the torus. We will call the dimension of the torus $n$. We write $\TT$ for the torus, or $\TT(\KK, \Lambda)$ when we want to specify the ground field $\KK$ and the lattice $\Lambda$.

Let $\Sigma$ be a complete rational fan in $\QQ \otimes \Lambda$ and let $\Toric(\Sigma)$ be the associated toric variety over an algebraically closed field $\KK$.  (See \cite{Fult} for background on toric varieties.) The open torus in $\Toric(\Sigma)$ is canonically $\Hom(\Lambda^{\vee}, \KK^*)$. For each ray (one dimensional cone) of $\Sigma$, there is a unique minimal element of $\Lambda$ on this ray; if we identify $\Lambda$ with $\ZZ^n$ then the element in question is the unique point of the ray whose coordinates are integers with no common factor. Let $\rho_1$, \dots, $\rho_N$ be the set of minimal vectors of the rays of $\Sigma$. The following is a special case of Theorem~3.1 of \cite{FultSturm}.
\begin{prop}
With the notation above, the Chow group $A^{n-1}$ is given by
$$A^{n-1}(\Toric(\Sigma)) \isomorph \{ (d_1, d_2, \ldots, d_N) \in \ZZ^N : \sum_{i=1}^N d_i \rho_i=0 \}. $$
This is a subgroup of $H^{2n-2}(\Toric(\Sigma), \ZZ)$, and equals $H^{2n-2}(\Toric(\Sigma),\ZZ)$ if $\Toric(\Sigma)$ is smooth.
\end{prop}

Let $\overline{X}$ be a smooth, complete algebraic curve and let $\phi: \overline{X} \to \Toric(\Sigma)$ be a map from $\overline{X}$ into $\Toric(\Sigma)$. Let $X$ denote $\phi^{-1}(\TT) \subset \overline{X}$, the part of $\overline{X}$ which is mapped to the big torus. Let us say that $(\phi,\overline{X})$ is \textbf{torically transverse} if $X$ is nonempty and $\phi(\overline{X})$ is disjoint from the toric strata of codimension $2$ and higher. (This is a specialization of the definition of torically transverse in \cite{SiebNish}.)

For each ray $\QQ_{\geq 0} \rho_i$ of $\Sigma$, let $Y_i$ be the codimension $1$ stratum of $\Toric(\Sigma)$ associated to $\rho_i$. Let $d_i$ be the length of $\phi^*(\OO_{Y_i})$, in other words, the number of points in $\phi^{-1}(Y_i)$ counted with multiplicity. Then $(d_1, \ldots, d_N)$ 
satisfies\footnote{In two paragraphs, we will see an alternate description of the $d_i$ which makes this clear.}  $\sum d_i \rho_i=0$ and hence corresponds to a class in $A^{n-1}(\Toric(\Sigma))$. Capping with the fundamental class gives the class in $A_1(\Toric(\Sigma))$, and hence in $H_2(\Toric(\Sigma))$, corresponding to $X$. Thus, if we want to study torically transverse curves representing a particular class in $H_2$, we may begin by finding the possible preimages of this class in $A^{n-1}$ and studying curves of that degree.

Continue to assume that $(\phi,\overline{X})$ is torically transverse. Let $x \in \overline{X} \setminus X$. We now describe how to determine which ray of $\Sigma$ corresponds to $x$, and with what multiplicity, solely by examining the map $X \to \TT$. Namely, for $\lambda \in \Lambda^{\vee}$, let $\chi^{\lambda}$ be the function on $\TT=\Hom(\Lambda^{\vee}, \KK^*)$ corresponding to $\lambda$. Then $\chi^{\lambda} \circ \phi$ is a function on $X$ and it extends to a meromorphic function on $\overline{X}$. Let $\sigma_x(\lambda)$ be the order of vanishing of this function at $x$. Then $\sigma_x$ is a linear map $\Lambda^{\vee} \to \ZZ$ and thus  an element of $\Lambda$. If $\sigma_x=0$ then $\phi$ extends to $x$, contradicting our choice that $x$ lies in $\overline{X} \setminus X$. Instead, $\sigma_x$ is a nonzero element of $\Lambda$. Write $\sigma_x$ as $d_x \rho_x$ where $d_x$ is a positive integer and $\rho_x$ is minimal. Then $\rho_x$ is the ray of $\Sigma$ corresponding to $x$ and $d_x$ is the multiplicity. 

We now make a definition: let $X$ be a smooth algebraic curve and let $\phi : X \to \TT$ be an algebraic map. (Whenever $\phi$ is nonconstant, the curve $X$ is not complete.) Let $\overline{X}$ be the smooth complete curve compactifying $X$; we impose the condition that $\phi$ can not be extended to any point of $\overline{X} \setminus X$. Let $x$ be any point of $\overline{X} \setminus X$. Define $\sigma_x$, $\rho_x$ and $d_x$ as before. Let $\rho_1$, \dots, $\rho_N$ be the set of distinct values of $\rho_x$ as $x$ ranges over $\overline{X} \setminus X$. For $1 \leq i \leq N$, let $d_i=\sum_{\rho_x=\rho_i} d_x$ and set $\sigma_i=\sum_{\rho_x=\rho_i} \sigma_x=d_i \rho_i$. We define the set $\{ \sigma_1, \ldots, \sigma_N \}$ to be the \textbf{degree} of $(\phi,X)$. Note that this is defined without any choice of toric compactification of $\TT$. Note also that we have $\sum \sigma_i=0$, because any rational function has equally many zeroes and poles on $\overline{X}$.

We then have:

\begin{prop} \label{ToricTransverse}
Let $\Sigma$ be a complete fan with rays generated by $\rho_1$, \ldots, $\rho_N$ and let $(d_1, \ldots, d_N)$ be any class in $A^{n-1}(\Toric(\Sigma))$. Then there is a bijection between torically transverse curves $(\phi,\overline{X})$ which represent the given class and maps $X \to \TT$ which are incapable of being extended to any larger compactficiation of $X$ and have degree $(d_1 \rho_1, \ldots, d_N \rho_N)$; this bijection is given by restriction to the preimage of $\TT$ in $\overline{X}$.
\end{prop}

For this reason, we can reformulate questions about constructing torically transverse curves in toric varieties into questions about constructing curves of given degree in torii. Ordinarily, of course, one wishes to consider all curves in some toric variety, not only the torically transverse ones. In many applications, it can be shown by dimensional considerations that all of the curves of interest are torically transverse. Even when this is not true, it is true that, if we specify the cohomology class of a curve in $\Toric(\Sigma)$, then there are only finitely many possible degrees for curves realizing that cohomology class and not lying in the toric boundary -- and those curves which do lie in the toric boundary are in the interiors of smaller toric varieties. For this reason, in this paper we will study problems where we specify the degree of a curve in the torus rather than specifying degrees in a toric compactification. 

\section{Basic Tropical Background}

Before introducing any algebraic technology, we discuss an issue concerning polyhedral terminology. In this paper, we will study a great number of polyhedra. All of these will be rational polyhedra, meaning that they are defined by finitely many inequalities of the form $\{ (x_1, \ldots, x_n) : \sum a_i x_i \leq \lambda \}$ where $(a_1, \ldots, a_n)$ is an integer vector and $\lambda$ is a rational number. It will be most convenient to consider these as subsets of $\QQ^n$. (For example, in Theorem~\ref{DefnsOfTrop}.) However, we want to be able to use topological language to talk about polyhedral complexes. We therefore adopt the following conventions: a polyhedron is a subset of $\QQ^n$, defined by finitely many inequalities as above. When we refer to a point of a polyhedron, we mean a point of $\QQ^n$. Nonetheless, when we describe a polyhedral complex using topological terms, such as ``connected'', ``simply connected'' and so forth, we will mean the properties of the closure of that complex in $\RR^n$. Similar issues will arise concerning metrized graphs. We adopt the conventions that the edges of a metrized graph always have rational lengths and that, the points of this graph, considered as a metric space, are the points which have rational distances from all of the vertices. Nevertheless, we will freely speak of graphs as connected, as mapping continuously from one to another, and so forth. 

Let $\OOO$ be a complete discrete valuation ring with valuation $v: \OOO \to \ZZ_{\geq 0} \cup \{ \infty \}$. We write $\KKK$ for the fraction field of $\OOO$ and $\kk$ for the residue field. We always assume that $\kk$ is algebraically closed. Let $\KK$ be the algebraic closure of $\KKK$, let $v: \KK \to \QQ \cup \{ \infty \}$ be the extension of $v$ to $\KK$, let $\OO$ be the ring  $v^{-1}(\QQ_{\geq 0} \cup \{ \infty \})$ and $\MM$ and $\MMM$ be the maximal ideals of $\OO$ and $\OOO$ respectively. Note that, since $\kk$ is algebraically closed, the residue field $\OO/\MM$ is also equal to $\kk$. As our notation suggests, we will primarily consider the objects $\KK$, $\OO$ and $\MM$ and only occasionally need to deal with $\KKK$, $\OOO$ and $\MMM$.\footnote{Several earlier tropical works, including some of my own work, attempted to ignore $\OOO$, $\KKK$ and $\MMM$ entirely. The reader should still view these objects as being only technicial crutches. However, I have become convinced that it is not worth trying to avoid them completely. The first problem is that, $\OO$ is not noetherian. This has never caused any irreperable difficulties, but it slows down work by making standard results inapplicable. Second, one of my motivations for avoiding the boldfaced objects was to be able to work with other fields with valuation, such as the field of those Puiseaux series over $\CC$ that have some positive radius of convergence. In the current paper, however, we will need to have a field complete with respect to $v$. I don't know of any useful examples where $\KK$ is complete that are not some sort of completion of the algebraic closure of some $\KKK$.} We fix a group homorphism $w \mapsto t^w$ from $\QQ \to \KK^*$ giving a section of $v: \KK^* \to \QQ$; this is always possible because $v$ is surjective and $\KK^*$ is divisible.\footnote{It is quite possible to do without this choice, and Sam Payne has advocated doing so. This is doubtless the most morally correct way of proceeding, but it introduces a great deal of notational baggage. In particular, we would be forced to replace torii with principal homogeneous spaces over torii in several places, and would thus no longer have explicit coordinates.} 

Let $f \in \KK[\Lambda^{\vee}]$ and let $w \in \QQ \otimes \Lambda$. We can write $f=\sum_{\lambda \in L} f_{\lambda} \chi^{\lambda}$, where $L$ is a finite subset of $\Lambda^{\vee}$ and each $f_{\lambda}$ is a nonzero element of $\KK$. Let $L_0$ be the subset of $L$ on which the function $\lambda \mapsto v(f_\lambda)+\langle \lambda, w \rangle$ achieves its minimum. (Here $\langle \ , \ \rangle$ is the pairing $\Lambda^{\vee} \times (\QQ \otimes \Lambda) \to \QQ$.) We define $\In_w(f)=\sum_{\lambda \in L_0} [t^{-v(f_\lambda)} f_{\lambda}]  \chi^{\lambda}$ where $g \mapsto [g]$ is the projection $\OO \to \OO/\MM = \kk$. (We set $\In_w(0)=0$.)

Suppose that we have an injection of rings $\kk \into \OO$ with $[\alpha]=\alpha$ for $\alpha \in \kk$. Suppose also that each coefficient $f_{\lambda}$ of $f$ lies in $\kk$. Then we can describe $\In_w(f)$ more simply as $\sum_{\lambda \in L_0} f_{\lambda} \chi^{\lambda}$; note that we sum over $L_0$, not over $L$. This is the definition that is used in Gr\"obner theory, where the field $\KK$ is left hidden in the background. Using this alternate definition, we can define $\In_w(f)$ for $f \in \kk[\Lambda^{\vee}]$ and $w \in \QQ \otimes \Lambda$, even without choosing a field $\KK$ to use. We will say that $\KK$ is a \textbf{power-series field} if we are given such a section $\kk \into \OO$. If $\KK$ is a power series field, we will say that a polynomial $f \in \KK[\Lambda^{\vee}]$ has \textbf{constant coefficients} if its coefficients lie in the image of $\kk$ and that a variety $X \subset \TT(\KK, \Lambda)$ has constant coefficients if it is defined by polynomials with constant coefficients.

If $I$ is an ideal in $\KK[\Lambda^{\vee}]$, let $\In_w(I)$ denote the ideal of $\kk[\Lambda^{\vee}]$ generated by $\In_w(f)$ for $f \in I$. If $X$ is the closed subscheme of $\TT(\KK, \Lambda)$ corresponding to $I$ then we write $\In_w X$ for the closed subscheme of $\TT(\kk, \Lambda)$ corresponding to $\In_w I$. The geometric meaning of $\In_w X$ is the following: let $t^w$ denote the element of $\TT(\KK, \Lambda)=\Hom(\Lambda^{\vee}, \KK^*)$ described by $t^w(\lambda)=t^{\langle \lambda, w \rangle}$ for $\lambda \in \Lambda^{\vee}$. Consider the subvariety $t^{-w} \cdot X$ of $\TT(\KK, \Lambda)$ where $\cdot$ denotes the standard action of $\TT(\KK, \Lambda)$ on itself. Let $\overline{t^{-w} \cdot X}$ be the (Zariski) closure of $t^{-w} \cdot X$ in $\Spec \OO[\Lambda^{\vee}]$. Then $\In_w X$ is the fiber of $\overline{t^{-w} X}$ over $\Spec \kk$. Moreover, suppose that $t^w \in \KKK$ and that $X$ is defined over $\KKK$, meaning that there is a subscheme $\boldsymbol{X}$ of $\TT(\KKK, \Lambda)$ such that $X=\boldsymbol{X} \times_{\KKK} \KK$. (Note that we can always achieve these hypotheses by replacing $\KKK$ with a finite extension.) Then we may instead describe $\In_w X$ by taking the Zariski closure of $t^{-w} \boldsymbol{X}$ in $\Spec \OOO[\Lambda^{\vee}]$.

If $I$ is an ideal of $\kk[\Lambda^{\vee}]$, rather than of $\KK[\Lambda^{\vee}]$, and $w$ any point of $\QQ \otimes \Lambda$ then we can define $\In_w I$ to be the ideal in $\kk[\Lambda^{\vee}]$ generated by $\In_w f$ for all $f \in I$, where $\In_w f$ is defined by the alternate definition two paragraphs above. While this is a slight abuse of notation, it should cause no confusion: one meaning of $\In_w I$ is defined for $I \subset \kk[\Lambda^{\vee}]$ and the other for $I \subset \KK[\Lambda^{\vee}]$ and the meanings are extremely closely related. Specifically, let $I$ be an ideal of $\kk[\lambda]$, let $w \in \QQ \otimes \Lambda$ and let $\KK'$ be any power seres field with associated notation $(\KK', v', \OO', \MM', \kk')$ . Suppose that $w \in \QQ \otimes \Lambda$ and that $\kk' = \kk$. Then $\In_w I=\In_w (\KK' \otimes_{\kk'} I)$.

The following lemma will be of frequent use:

\begin{lemma} \label{PerturbInit}
Let $I$ be an ideal of $\KK[\Lambda^{\vee}]$. Let $w$ and $v$ be elements of $\QQ \otimes \Lambda$. Then, for any sufficiently small rational number $\epsilon$, we have $\In_{w+\epsilon v} I = \In_v \In_w I$.
\end{lemma}

\begin{proof}
This is proved in a slightly less general context as Proposition~1.13 in \cite{GB+CP}; the proof can be adapted to our setting.
\end{proof}

We now define tropicalization.

\begin{prop} \label{DefnsOfTrop}
Let $X$ be a closed subscheme of $\TT(\KK, \Lambda)$ and let $I \subset \KK[\Lambda^{\vee}]$ be the corresponding ideal. Let $w \in \QQ \otimes \Lambda$. Then the following are equivalent:
\begin{enumerate}
\item There is a point $x \in X(\KK)$ with $v(x)=w$.
\item There is a valuation $\tilde{v} : \KK[\Lambda^{\vee}]/I \to \QQ \cup \{ \infty \}$ extending $v : \KK \to \QQ \cup \{ \infty \}$ with the property that $v(\chi^{\lambda})=\langle \lambda, w \rangle$ for every $\lambda$ in $\Lambda$. 
\item For every $f \in I$, the polynomial $\In_w f$ is not a monomial.
\item The ideal $\In_w I$ does not contain any monomial.
\item The scheme $\In_w X$ is nonempty.
\end{enumerate}
\end{prop}

Note that, in conditions (3) and (4), zero is not considered a monomial.

\begin{proof}
The equivalence of (1), (3) and (4) is Theorem~2.1 of \cite{SpeySturm}. A very careful and detailed version of this result, by reduction to the hypersurface case, is in \cite{Payne}. The equivalence of (4) and (5) is simply the Nullstellansatz -- since monomials are units in $\kk[\Lambda^{\vee}]$, the ideal $\In_w I$ contains a monomial if and only if it is all of $\kk[\Lambda^{\vee}]$. The equivalence of (2) and (5) is Theorem~2.2.5 of \cite{EKL}. 
\end{proof}

Define the subset of $\QQ \otimes \Lambda$ where any of the equivalent conditions above holds to be $\Trop X$. Note that condition (3) of the above proposition clearly singles out a closed set.  If $\KK$ is a power series field and $X$ has constant coefficients, then we can define $\Trop X \subset \RR^n$ using the definition of $\In_w f$ which is defined for $f \in \kk[\Lambda^{\vee}]$. There is little risk of confusion in defining $\Trop X$ both for $X \subset \TT(\KK, \Lambda)$ and $X \subset \TT(\kk, \Lambda)$. The precise relation is as follows: Let $\KK'$ be any power series field, with associated notation $(\KK', v', \OO', \MM', \kk')$ such that $\kk' = \kk$ and let $X$ be a closed subscheme of $\TT(\kk, \Lambda)$. Then $\Trop X=\Trop (X \times_{\Spec \kk'} \Spec \KK')$. The next proposition summarizes basic results on the structure of $\Trop X$.

\begin{prop}
Let $X \subset \TT(\KK, \Lambda)$. Then $\Trop X$ can be given the structure of a polyhedral complex (with finitely many faces). Furthermore, we may do this in such a way that, for $\sigma$ any face of this polyhedral complex and $w$ and $w'$ two points in the relative interior of $\sigma$, we have $\In_w X=\In_{w'} X$. If $X$ is $d$-dimensional then $\Trop X$ has dimension $d$. If $X$ is pure of dimension $d$ then so is $\Trop X$. If $X$ is connected then $\Trop X$ is connected. If $X$ is connected in codimension $1$ and $\KK$ has characteristic zero then $\Trop X$ is connected in codimension $1$.

If $\KK$ is a power series field and $X$ has constant coefficients, then $\Trop X$ can be given the structure of a polyhedral fan.
\end{prop}

We will call a polyhedral subdivision of $\Trop X$ which is as described above a \textbf{good} subdivision.

\begin{proof}
The existence of the polyhedral structure is proved by Bieri and Groves \cite{BG} using description (2) of $\Trop X$.  The claim about initial ideals is proved by Sturmfels in a slightly more specialized context in the course of proving \cite[Theorem 9.6]{SPE}. 
%Here, Sturmfels uses definition (4) of $\Trop X$. 
We will reprove this porperty below in our context in Corollary~\ref{GoodSubDiv}. The dimensionality claim is also proven in Bieri and Groves \cite{BG} and is proven by a different method (under more restrictive hypotheses than we adopt here) by Sturmfels \cite{SPE}. Sturmfels proves the pureness claim as well. The connectivity result result is proven in \cite{EKL}.
% using a mix of descriptions (1) and (2). 
The connectivity in codimension one is proven in \cite{CTV}. That proof is given in a somewhat more restrictive setting then we have adopted here, but there is no difficulty in extending the arguments. 
\end{proof}

\section{Statement of Results} \label{Results}

We now have enough tropical background to state our main results. We need one combinatorial definition:

Let $\Gamma$ be a finite graph. We write $\partial \Gamma$ for the set of degree $1$ vertices of $\Gamma$. Let $\iota$ be a continuous map $\Gamma \setminus \partial \Gamma \to \QQ \otimes \Lambda$ such that an edge $e$ of $\Gamma$ is taken to
\begin{enumerate}
\item either a finite line segment or a point if neither endpoint of $e$ is in $\partial \Gamma$,
\item an unbounded ray if one endpoint of $e$ is in $\partial \Gamma$ and
\item a line if both ends of $e$ are in $\partial \Gamma$.
\end{enumerate}
We consider such pairs $(\iota,\Gamma)$ up to reparameterization of the edges of $\Gamma$. We require that $\iota(e)$ has slope in $\Lambda$ for every edge $e$ of $\Gamma$.  If $v$ is a vertex of $\Gamma$, and $e$ is an edge of $\Gamma$ with an endpoint at $v$, then we write $\rho_v(e)$ for the minimal lattice vector parallel to $\iota(e)$ which points in the direction away from $\iota(v)$. If $e$ is mapped to a point, define $\rho_v(e)$ to be $0$. Suppose that $m$ is a function assigning a positive integer to each edge of $\Gamma$. We say that $(\iota,\Gamma,m)$ is a \textbf{zero tension curve} if, for every vertex $v$ in $\Gamma \setminus \partial \Gamma$, we have $\sum_{e \ni v} m(e) \rho_v(e)=0$. We introduce the notation $\sigma_v(e)$ for $m(e) \rho_v(e)$. If $(\iota, \Gamma, m)$ is any zero tension curve, we place a metric\footnote{This might be only a pseudo-metric; if $\iota$ collapses an edge of $\Gamma$ then we have $d(x,y)=0$ for some $x$ and $y$ which are not equal. This will not be a difficulty.} on $\Gamma$ by setting the edge $e$ of $\Gamma$ to have length $\ell$, where the endpoints of $\iota(e)$ differ by $\ell \rho_v(e)$.

We define a \textbf{partial zero tension curve} to be a triple $(\iota, \Gamma, m)$ which obeys the above conditions except that, if one endpoint of an edge $e$ is in $\partial \Gamma$, we permit that edge to be taken to a finite line segment rather than a ray. Partial zero tension curves are analogous to analytic maps from a Riemann surface with holes of positive area; zero tension curves, which will be our main concern, are analogous to algebraic maps from punctured Riemmann surfaces.

We define the genus of a zero tension curve to be the first Betti number of $\Gamma$. We define the degree of a zero tension curve as follows: Let $D \subset \Lambda$ be the (finite) set of values assumed by $-\rho_v(e)$  as $v$ ranges through $\partial \Gamma$. For each $\lambda \in D$, let $m_{\lambda}=\sum m(e)$ where the sum is over $e$ with an endpoint $v$ in $\partial \Gamma$ and $- \rho_v(e)=\lambda$. Then the degree of $(\iota, \Gamma,m)$ is the set $\{ m_{\lambda} \cdot \lambda: \lambda \in D \}$. We now state that, given a curve $X$, the polyhedral complex $\Trop X$ reflects the degree and genus of the curve.

\begin{Theorem} \label{NecessaryConditions}
Let $X$ be a connected (punctured) curve of genus $g$ over $\KK$ equipped with a map $\phi : X \to \TT(\KK, \Lambda)$. Let $\delta \subset \Lambda$ be the degree of $(X,\phi)$. Then there is a connected zero tension curve $(\iota, \Gamma,m)$ of degree $\delta$ and genus at most $g$ with $\iota(\Gamma)=\Trop \phi(X)$.

Extend $X$ to a flat family over $\Spec \OO$ whose fiber over $\Spec \kk$ consists of smooth reduced curves glued along nodes.\footnote{For example, extend $X$ to a family of stable curves.} If any of the components of the $\kk$-fiber are not rational, then we can take the genus of $\Gamma$ to be strictly less than $g$.
\end{Theorem}

Theorem~\ref{NecessaryConditions} is implicit in the work of many authors, beginning with Grigory Mikhalkin. A complete proof is given in \cite{SiebNish}; we explain how to find this theorem in that work: Proposition~6.3 of \cite{SiebNish} states that, given $X$ and $\phi$, there is degeneration of $\TT(\KK, \Lambda)$ (over $\Spec \OO$) to a union of toric varieties and a degeneration of $X$ (over $\Spec \OO$) to a nodal curve so that $\phi$ extends on the $\Spec \kk$ fiber to a torically transverse stable map. Write $X_0$ for the nodal curve and $\phi_0$ for the map from $X_0$. In the course of proving Theorem~8.3, Nishinou and Siebert verify that $(X_0, \phi_0)$ is an object they call a pre-log curve. In Construction~4.4, they explain how to build a zero tension curve from a pre-log curve. 

In the appendix, we give a proof of Theorem~\ref{NecessaryConditions} on the assumption that $\In_w X$ is reduced for every vertex $w$ of $\Trop X$. This gives us an opportunity to demonstrate a number of general tropical tools that should be recorded for general use. Removing the reducedness hypothesis would probably require working with stable maps rather than subvarieties throughout and introduces a number of technicalities -- most likely, there is no better way to handle this problem than to mimic Siebert and Nishinou's argument.

We now state the main  results of the paper.

\begin{Theorem} \label{GenusZeroExists}
Let $(\iota, \Gamma, m)$ be a zero tension curve of genus zero. Then there is a (punctured) genus zero curve $X$ over $\KK$, and a map $\phi : X \to \TT(\KK, \Lambda)$ so that $(X, \phi)$ has degree $\delta$ and $\iota(\Gamma)=\Trop \phi(X)$.
\end{Theorem}

Theorem~\ref{GenusZeroExists} was previously proven, by methods of log geometry, in \cite{SiebNish}. 

\begin{Theorem} \label{GenusOneExists}
Let $(\iota, \Gamma, m)$ be a zero tension curve of genus one and degree $\delta$. Let $e_1$, \dots, $e_r$ be the edges of the unique circuit of $\Gamma$. Assume that the slopes of $\iota(e_1)$, \dots, $\iota(e_r)$ span $\QQ \otimes \Lambda$. Then there is a (punctured) genus one curve $X$ over $\KK$, and a map $\phi : X \to \TT(\KK, \Lambda)$ so that $(X, \phi)$ has degree $\delta$ and $\iota(\Gamma)=\Trop \phi(X)$.
\end{Theorem}

The importance of the criterion that the slopes of the edges $\iota(e_1)$, \dots, $\iota(e_r)$ span $\QQ \otimes \Lambda$ was first pointed out by Mikhalkin. Following Mikhalkin, we say that $(\iota, \Gamma,m)$ is \textbf{ordinary} when this condition holds, and \textbf{superabundant} when it does not. When dealing with superabundant curves, we need to impose a further criterion, which is original to this paper. Let $(\iota, \Gamma, m)$ be a zero tension curve of genus one and degree $\delta$. Let $e_1$, \dots, $e_r$ be the edges of the unique circuit of $\Gamma$. Let $H$ be an affine hyperplane in $\QQ \otimes \Lambda$ containing all of the line segments $\iota(e_i)$. Let $\Delta$ be the connected component of $\Gamma \cap \iota^{-1}(H)$ which contains the circuit of $\Gamma$ and let $x_1$, \dots, $x_s$ be the vertices of $\Delta$ which are also in the (topological) closure of $\Gamma \setminus \Delta$; we call these the \textbf{boundary vertices of $\Delta$}. Let $d_1$, \dots, $d_s$ be the distances from $x_1$, \dots $x_s$ to the nearest point on the circuit of $\Gamma$. Then we say that $(\iota, \Gamma,m)$ is \textbf{well spaced with respect to $H$} if the minimum of the numbers $(d_1, \dots, d_s)$ occurs more than once. We say that $(\iota, \Gamma, m)$ is \textbf{well spaced} if it is well spaced with respect to every $H$ containing the circuit of $\Gamma$. 

In \cite{Mikh2}, Mikhalkin states without proof a result (Theorem 1) which includes both Theorems~\ref{GenusZeroExists} and \ref{GenusOneExists}. It is my understanding that this proof will appear shortly.

\begin{Theorem} \label{GenusOneSuperAbundant}
Assume that $\kk$ has characteristic zero. Let $(\iota, \Gamma, m)$ be a zero tension curve of genus one and degree $\delta$ and assume that $(\iota, \Gamma, m)$ is well spaced. Then there is a (punctured) genus one curve $X$ over $\KK$, and a map $\phi : X \to \TT(\KK, \Lambda)$ so that $\iota(\Gamma)=\Trop X$.
\end{Theorem}

There is a partial converse to this theorem, see Theorem~\ref{Necessity}.

It is possible to prove enumerative versions of all of these results, where we count curves $(X, \phi)$ with $\Trop \phi(X)=\iota(\Gamma)$ that meet subvarieties of $\TT(\KK, \Lambda)$. We do not do so here. Partially, this is because it would add greatly to the length of the exposition. A more important reason is that, as described in the introduction, we have no results regarding the lifting of zero tension curves of genus zero to actual curves of genus one and therefore we do not know how to productively apply such enumerative results.

I intend to write a sequel to this paper, proving analogous results for curves of genus greater than $1$.

\section{The Bruhat-Tits Tree} \label{BTTree}

We will spend the rest of this paper proving Theorems~\ref{GenusOneExists} and \ref{GenusOneSuperAbundant}. One of our main technical tools is the Bruhat-Tits tree. A good
reference our discussion of the Bruhat-Tits tree is Chapter 2 of \cite{MS}.

We denote by $\BT(\KK)$ the set of $\OO$-submodules of $\KK^2$ which are
isomorphic to $\OO^2$, modulo $\KK^*$-scaling. We write $\overline{M}$ for the equivalence class 
of a module $M$. We equip $\BT(\KK)$ with the
metric where $d(\overline{M_1}, \overline{M_2})$ is the minimum of the set of $\epsilon$ such that
there exists an $\alpha$ with $M_1 \supseteq t^{\alpha} M_2 \supseteq
t^{\alpha + \epsilon} M_1$; this minimum exists and is independent of the choice
of representatives $M_1$ and $M_2$. Clearly, $d(\overline{M_1}, \overline{M_2})$ is always in $A=v(\KK^*)$. The metric space $\BT(\KK)$ is called the Bruhat-Tits
tree of $K$.

If we made the analogous construction working over $\KKK$ we could equip $\BT$ with the
structure of the vertices of a tree so that distance was the graph
theoretic distance. Instead, $\BT(\KK)$ is what is called an $\QQ$-tree
(see \cite{MS}). We remind our reader of the convention that all metric trees have edges whose lengths are in $\QQ$, and the points of such a tree are the points whose distances from the vertices are rational. Being a $\QQ$-tree is a more general concept than this; the following proposition lists the ``tree-like'' properties of $\QQ$.

\begin{prop}
If $\overline{M_1}$ and $\overline{M_2} \in \BT(\KK)$ with $d(\overline{M_1}, 
\overline{M_2})=d$ then there is a
unique distance preserving map $\phi : [0,d] \cap \QQ \to \BT(\KK)$ with
$\phi(0)=\overline{M_1}$ and $\phi(d)=\overline{M_2}$. Explicilty, if 
$t^{\alpha} M_1 \supseteq M_2 \supseteq t^{\alpha+d} M_1$ then $\phi(e)=\overline{t^{\alpha+e} M_1 + M_2}$ We will call the image of $\phi$ the
path from $\overline{M_1}$ to $\overline{M_2}$ and denote it by $[\overline{M_1}, 
\overline{M_2}]$. If $\overline{M_1}$,
\dots, $\overline{M_n} \subset \BT(\KK)$ then $\cup_{i \neq j} [\overline{M_1}, 
\overline{M_2}]$ is a metric tree.
\end{prop}

Suppose
now that $(x_1:y_1)$ and $(x_2:y_2)$ are distinct members of
$\PP^1(\KK)$. Then we can similarly define a map $\phi: \QQ \to \BT(\KK)$
by $\phi(e)=\overline{\OO (x_1,y_1)+ t^e \OO(x_2, y_2)}$. We will call the image of
this $\phi$ the \textbf{path from $(x_1:y_1)$ to $(x_2:y_2)$} and denote
it $[(x_1:y_1), (x_2:y_2)]$. Similarly, if $\overline{M} \in \BT(\KK)$ and
$(x:y) \in \PP^1(\KK)$, we can define a semi-infinite path from
$(x:y)$ to $M$ denoted $[(x:y),\overline{M}]$.

If $Z$ is a subset of $\BT(\KK) \cup \PP^1(\KK)$, we denote by $[Z]$ the subspace $\cup_{z, z' 
\in Z} [z,z']$ of $\BT(\KK)$. For simplicity, assume that $|Z| \geq 3$. If $Z$ is finite, 
then $[Z]$ is a metric tree with a semi-infinite ray for each member of $Z \cap \PP^1(\KK)$. We will say 
that this ray has its end at the corresponding member of $Z \cap \PP^1(\KK)$. We will 
abbreviate $[\{ z_1, \ldots, z_n \}]$ as $[z_1, \ldots, z_n]$.

The particular case where $Z$ is a four element subset of $\PP^1(\KK)$ is of particular 
importance for us. Let $\{ w,x,y,z \} \subset \PP^1(\KK)=\KK \cup \{ \infty \}$. We define the 
\textbf{cross ratio} $c(w,x:y,z)$ by
$$c(w,x:y,z)=\frac{(w-y)(x-z)}{(w-z)(x-y)}.$$
Note that $c(w,x:y,z)=c(x,w:z,y)=c(y,z:w,x)=c(z,y:x,w)$ and $c(w,x:y,z)=c(w,x:z,y)^{-1}$.

\begin{prop} \label{CrossRatio}
The metric space $[ w,x,y,z ]$ is a metric tree with $4$ semi-infinite rays and either $1$ or $2$ internal vertices. 

If $[ w,x,y,z ]$ has $2$ internal vertices, let $d$ be the length of the internal edge and 
suppose that the rays ending at $w$ and $x$ lie on one side of that edge and the rays through 
$y$ and $z$ on the other. Then $v(c(w,x:y,z))=0$, $v(c(w,y:x,z))=-v(c(w,y:z,x))=d$ and $v(c(w,z:x,y))=-v(c(w,z:y,x))=d$. The first statement can be strengthened to say that $v(c(w,x:y,z)-1)=d$.
(The valuations of all other permutations of $\{ w,x,y,z \}$ can be deduced from these by the symmetries of the crossratio.)

If $[\{ w,x,y,z \}]$ has only $1$ internal vertex then $v(c(w,x:y,z)-1)=v(c(w,x:y,z))=0$ and the same holds for all permutations of $\{ w,x,y,z \}$.
\end{prop}

This proposition can be remembered as saying ``$v(c(w,x:y,z))$ is the signed length of $[w,x] 
\cap [y,z]$'' where the sign tells us whether the two paths run in the same direction or the 
opposite direction along their intersection.

\begin{proof}
The group
$\GL_2(\KK)$ acts on $\BT(\KK)$ through the action on $\KK^2$.  This
action is compatible with the standard action of $\PGL_2(\KK)$ on
$\PP^1(\KK)=\KK \cup \{ \infty \}$. It is well known that $c$ is
$\PGL_2(\KK)$ invariant. So the whole theorem is invariant under
$\PGL_2(\KK)$ and we may use this action to take $w$, $x$ and $y$ to
$0$, $1$ and $\infty$. Our hypothesis in the second paragraph is that
$[0,1,\infty, z]$ is a tree with $0$ and $1$ on one side of a finite
edge of length $d$ and $z$ and $\infty$ on the other. It is easy to
check that this is equivalent to requiring that $v(z)=-d<0$. Then
$c(0,1:\infty,z)=1-1/z$ which does indeed have
valuation $0$ and $c(0,1:\infty,z)-1=-1/z$ does indeed have valuation $d$. Similarly, $c(0,\infty:1,z)=1/z$ which has valuation $d$ and
$c(0,z:1,\infty)=1/(1-z)$ which has valuation $d$. In the second paragraph of the proposition,
the assumption that the tree has no finite edge implies that
$v(z)=v(z-1)=0$ and the argument then continues as before.
\end{proof}

\section{Lemmas on Zero Tension Curves} \label{ZTLemma}

We pause to prove two combinatorial lemmas about zero tension curves. 

\begin{Lemma} \label{LocalMaxLemma}
Let $(\iota, \Gamma,m)$ be a connected partial \footnote{Recall that the adjective partial means we permit edges which end at a degree one vertex of $\Gamma$ to be taken to a finite line segment rather than to an infinite ray.} zero tension curve in $\QQ^n$. Suppose that $\iota(\Gamma)$ is not contained in any hyperplane. Then the set of vectors $\sigma_v(e)$, where $v$ runs over the degree one vertices  of $\Gamma$, spans $\QQ^n$.
\end{Lemma}

\begin{proof}
Suppose, for the sake of contradiction that there is some nonzero $\lambda \in \QQ^n$ with $\langle \lambda, \sigma_v(e) \rangle=0$ for every degree one vertex $v$ of $\Gamma$. Let $h(u)$ be the function $\langle \lambda, \iota(u) \rangle$ on $\Gamma \setminus \partial \Gamma$. As $h$ is constant on every ray of $\Gamma$ ending at a degree one vertex, and in particular on all of the unbounded rays of $\Gamma$, the function $h$ is bounded on $\Gamma$. Let $U$ be the (nonempty) subset of $\Gamma$ on which $h$ achieves its maximum. Clearly,  $U$ is closed.

We now show that $U$ is also open. Clearly, if $U$ contains a point $p$ in the interior of an edge of $\Gamma$, then it contains that entire edge and, in particular, $U$ contains an open neighborhood of $p$. So we just need to show that, if $u$ is a vertex of $\Gamma$ contained in $U$ then $U$ contains an open neighborhood of $u$. If $u$ is a degree one vertex of $\Gamma$ then $h$ is constant in a neighborhood of $u$ by our hypothesis. If $u$ is not a degree one vertex then, by the zero tension condition, we have $\sum_{e \ni u} \langle \lambda, \sigma_u(e) \rangle =0$. Since we have assumed that $h$ is maximized at $u$, we have  $\langle \lambda, \sigma_u(e) \rangle \leq 0$ for every edge $e$ containing $u$ and we conclude that $h$ is constant in a neighborhood of $u$, as desired. 

So $U$ is open, closed and nonempty. As $\Gamma$ is connected, $U=\Gamma$ and we have that $h$ is constant on $\Gamma$. This contradicts our assumption that $\iota(\Gamma)$ is not contained in any hyperplane.
\end{proof}

One difficulty with Theorem~\ref{NecessaryConditions} is that it states that a zero tension curve with $\iota(\Gamma)=\Trop X$ exists, but it doesn't help us choose from among several possible candidates for $(\iota, \Gamma, m)$. We now introduce a concept that will let us guarantee that essentially only such $(\iota, \Gamma, m)$ exists.  We define $(\phi, X)$ to be \textbf{trivalent} if, for every edge $e$ of (some, equivalently any, good subdivision of) $\Trop \phi(X)$ we have $m_e=1$ and, for any vertex $v$ of (some, equivalently any, good subdivision of) $\Trop \phi(X)$ the degree of $v$ is at most $3$.

\begin{Lemma} \label{trivalent}
Let $(\phi, X)$ be a trivalent curve. Let $(\iota, \Gamma,m)$ be a zero tension curve with $\iota(\Gamma)=\Trop \phi(X)$ such that, if $e$ is any edge of (a good subdivision of) $\Trop \phi(X)$, then we have $m_e=\sum_{f \in \iota^{-1}(e)} m_f$. 

Than there is a subgraph $\Gamma'$ of $\Gamma$ which maps isomorphically onto $\Trop \phi(X)$. Specifically, we take $\Gamma'$ to be the union of all edges of $\Gamma$ which are not contracted to a point under $\iota$.
\end{Lemma}

\begin{proof}
First, we note that if $e$ is an edge of $\Trop \phi(X)$ then the interior of $e$ can have only one preimage in $\Gamma$, by the equation $1=m_e=\sum_{f \subset \iota^{-1}(e)} m_f$. Next, let $x$ be a vertex of $\Trop \phi(X)$ with edges $e_1$, $e_2$ and possibly $e_3$ coming out of $x$. Let $y \in \Gamma'$ be a preimage vertex of $x$. Then there must be edges leaving $y$ which map down to each of the $e_i$, as otherwise the zero tension condition would be violated. (If all of the edges leaving $y$ map down to a point then $y$ is not in $\Gamma'$.) Then there can be no other vertex $z$ of $\Gamma'$ which maps to $x$, as there are no edges of $\Trop \phi(X)$ left for the edges coming from $z$ to map to.

So every vertex of $\Trop \phi(X)$ and the interior of every edge of $\Trop \phi(X)$ has only one preimage in $\Gamma'$; \emph{i.e.} the restriction of $\iota$ to $\Gamma'$ is bijective onto its image. Moreover, $\Gamma'$ is closed in $\Gamma$, so the map $\Gamma' \to \Trop \phi(X)$, like the map $\Gamma \to \Trop \phi(X)$, is a closed map. But we know  $\Gamma' \to \Trop \phi(X)$ is bijective, so this is also an open map. A continuous open bijective map is a homeomorphism.
\end{proof}

\section{Tropical Curves of Genus Zero} \label{GenusZeroProof}

The aim of this section is to prove Theorem~\ref{GenusZeroExists}. This result
will appear in a future publication of Mikhalkin; it also appears with
many results on incidence conditions in \cite{SiebNish}. Our method of proof
is not only more explicitly constructive than these, but will also preview many of the methods which we will use to deal with higher genus curves. Let $(\iota, \Gamma, m)$ be a zero tension curve with
$\Gamma$ a tree. Put a metric on $\Gamma$ in the following manner: Give the unbounded
edges of $\Gamma$ length infinity. If $e$ is a finite edge running from $u$ to $v$, and $\iota(v)=\iota(u)+\ell \sigma_u(e)$, then give $e$ length $\ell$ Fix an identification of $\Lambda$ with $\ZZ^n$.

\begin{prop} \label{FindATree}
Let $T$ be a metric tree with finitely many vertices. Then there is a subset $Z$ of $\PP^1(\KK) \cup \BT(\KK)$ such that $[Z]$ is isometric to $T$. If every leaf of $T$ is at the end of an unbounded edge, then we can take $Z \subset \PP^1(\KK)$.
\end{prop}

\begin{proof}
%First, we consider the case where  each degree $1$ vertex of $T$ 
%is at the end of an infinite ray.

Our proof is by induction on the number of finite length edges of $T$. If $T$
has $l \geq 3$ leaves and no finite edges then $T$ is isometric to $[z_1,
\ldots, z_l]$ for $\{ z_1, \ldots, z_l \}$ any $l$ elements of $\KK^*$
with valuation $0$ and distinct images in $\kk^*$. (Such elements exist because $\kk$ is algebraically closed and hence infinite.) If $T$ is an unbounded edge which is infinite in both directions then $T \isomorph [0, \infty]$; if $T$ is an unbounded ray which is infinite in one direction then $T=[\OO^2, \infty]$; if $T$ is a point then $T \isomorph [\OO^2]$. These are all of the cases with no finite edges, so our base case is complete.

Now, let $e$ be a finite edge of $T$ of length $d$ joining vertices
$v_1$ and $v_2$. Remove $e$ from $T$, separating $T$ into two trees
$T_1$ and $T_2$. 
Define trees $T'_s$, where $s=1$, $2$, by adding an
unbounded edge to $T_s$ at $v_s$. 
By induction, we can find subsets
$Z_1$ and $Z_2 \subset \PP^1(\KK) \cup \BT(\KK)$ with $[Z_s]$ isometric to $T'_s$. Let
$z_s \in Z_s$ be the element of $Z_s$ at the end of the new ray added
to $T_s$. Without loss of generality, we may assume that $z_1=0$ and
$z_2=\infty$. Then the point of $T_s$ corresponding to $v_s$ lies
somewhere on $[0,\infty]$. By multiplying $Z_1$ and $Z_2$ by elements
of $\KK^*$, we may assume that these points lie distance $d$ apart with
$v_2$ closer to $z_1$ than $v_1$ is. Then $T$ is isometric to $[
(Z_1 \setminus \{ z_1 \}) \cup ( Z_2 \setminus \{ z_2 \})]$.

If all of the leaves of $T$ were at the end of unbounded edges, then this would also be true of $T'_1$ and $T'_2$, and tracing through the proof we see that $Z \subset \PP^1(\KK)$.
\end{proof}

Recal that we have been given a zero tension curve $(\iota, \Gamma, m)$ of genus zero and we want to construct an actual genus zero curve with $\iota(\Gamma)$ as its tropicalization. Let $Z \subset \PP^1(\KK)$ be such that $[Z]$ is isometric to
$\Gamma$. We define multisets $Z^+_1$, \dots, $Z^+_n$, $Z^-_1$, \dots,
$Z^-_n$ as follows: All of the elements of $Z^{\pm}_i$ lie in $Z$. Let
$z \in Z$ correspond to the end of an infinite ray $e$ of
$\Gamma$. Suppose that $\sigma_z(e)=(s_1,\ldots, s_n)$. Then $z \in
Z^{\pm}_i$ if and only if $\pm s_i<0$. In this case, the
number of times that $z$ occurs in $Z^{\pm}_i$ is $|s_i|$. Let $\phi_i$ be a rational function on $\PP^1$ with zeroes at the points of $Z^{+}_i$ and poles at the points of $Z^{-}_i$. Assuming that $\infty$ is not in $Z^{+}_i$ or $Z^{-}_i$, we may take $\phi_i(u)= \prod_{z \in Z_i^{+}} (u-z) /
\prod_{z \in Z_i^{-}} (u-z)$. Define a rational map
$\phi: \PP^1(\KK) \to \KK^n$ by the formula $\phi(u)=(\phi_1(u), \ldots, \phi_n(u))$.
Here $u$ is a coordinate on $\PP^1(\KK)$, thought of as $\KK \cup \{ \infty \}$.

We now show that we have indeed built a genus zero curve with tropicalization $\iota(\Gamma)$. Thus, when we prove the following theorem, we will complete the proof of Theorem~\ref{GenusZeroExists}.

\begin{Theorem} \label{GenusZeroFundamentalComputation}
The curve $\phi(\PP^1(\KK))$ is a genus zero curve of degree the degree of $(\iota, \Gamma, m)$.  $\Trop \phi(\PP^1(\KK))$ is a translation of $\iota(\Gamma)$. Furthermore, $\Trop \phi(\KK)$ is a translation of $\iota(\Gamma)$. Thus, by rescaling the $\phi_i(u)$ by elements of $\KK^*$, we can arrange that $\Trop \phi(\PP^1(\KK))$ is $\iota(\Gamma)$ 
\end{Theorem}

From now until the 
end of the proof, we identify $[Z]$ with $\Gamma$ so that we can write $\iota : [Z] \to 
\QQ^n$.

\begin{proof}
The curve $\phi(\PP^1(\KK))$ has genus zero because it is defined as the image of a genus zero curve. It has the same degree as $(\iota, \Gamma, m)$ because, from Proposition~\ref{ToricTransverse}, the degree can be computed simply by looking at the orders of vanishing of the coordinate functions on $(\KK^*)^n$ at the points of $\PP^1$ where $\phi$ is not defined. We built $\phi$ to have exactly the required zeroes and poles. We now move to the interesting point, the claim that $\Trop \phi(\PP^1(\KK))$ is a translation of $\iota(\Gamma)$. 

Let $u \in \PP^1(\KK) \setminus Z$. Then $[Z]$ is a tree and $[Z \cup \{
u \}]$ is a tree with one additional end. Let $b(u)$ be the point of $[Z]$
at which that end is attached. (See Figure~\ref{BOfU}. Here $[Z]$ is shown in solid lines, the points of $Z$ are represented by $z$'s, the path from $u$ to $b(u)$ is dashed and $b(u)$ is the solid dot.) We claim that, up to a translation,
$v(\phi(u))$ is $\iota(b(u))$. In other words, if $u_1$ and $u_2$ are distinct
members of $u \in \PP^1(\KK) \setminus Z$, we must show that for each
$i$ between $1$ and $n$ we have
$$v(\phi_i(u_1))-v(\phi_i(u_2))=\iota(b(u_1))_i - \iota(b(u_2))_i.$$

\begin{figure}
\centerline{\includegraphics{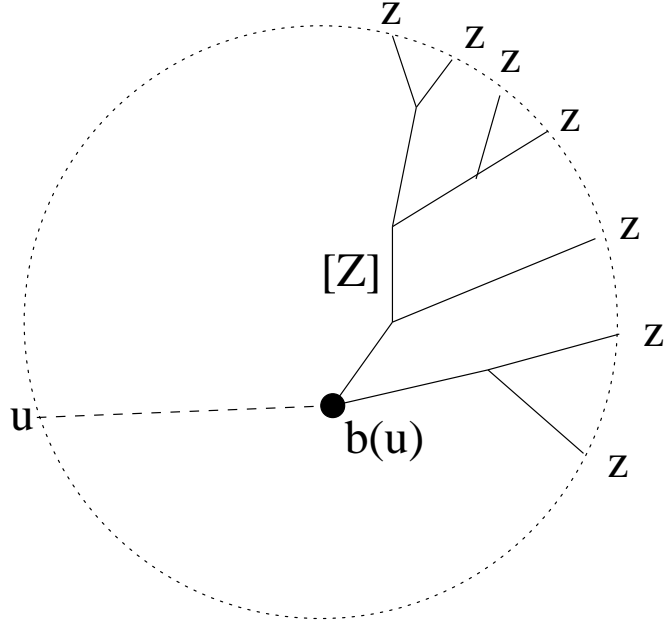}}
\caption{The Definition of $b(u)$} \label{BOfU}
\end{figure}

It is enough to show this in the case where $b(u_1)$ and $b(u_2)$ lie in the same edge $e$ of 
$[Z]$. Let $\sigma_{u_2}(e)=(s_1,\cdots s_n)$. We will fix one coordinate $i$ to pay attention to, so $i$ will not appear in our 
notation. Let $Z_i^{+}= \{ z_1^{+}, \ldots, z_r^{+} \}$ and $Z_i^{-}= \{ z_1^{-}, \ldots, 
z_r^{-} \}$. We may find constants $1 \leq s^+, s^{-} \leq n$ and order the $z_j^{\pm}$ such 
that $z_j^{\pm}$ is on the $b(u_1)$ side of $e$ for $1 \leq j \leq s^{\pm}$ and on the 
$b(u_2)$ side of $e$ for $s^{\pm} +1 \leq j \leq r$. Let $d$ be the distance from $b(u_1)$ to 
$b(u_2)$.

We have 
\begin{eqnarray*}
v(\phi_i(u_1))-v(\phi_i(u_2)) &=& v \left( \frac{\phi_i(u_1)}{\phi_i(u_2)} \right) \\
&=& v \left( \frac{ \left( \prod_{j=1}^r (u_1-z_j^{+}) / \prod_{j=1}^r (u_1-z_j^{-}) \right) 
}{ \left( \prod_{j=1}^r (u_2-z_j^{+}) / \prod_{j=1}^r (u_2-z_j^{-}) \right) } \right) \\
&=& v \left( \prod_{j=1}^r c(u_1,u_2: z_j^{+}, z_j^{-}) \right) \\
&=& \sum_{j=1}^r v( c(u_1,u_2: z_j^{+}, z_j^{-})) = d (s^+ - s^- ).
\end{eqnarray*}
The last equality is by applying Proposition~\ref{CrossRatio} to each term.

By the zero tension condition, $s_i=s^+ - s^-$. So $\iota(b(u_1))_i-\iota(b(u_2))_i$ is also $d(s^+ - s^-)$.
\end{proof}

We pause for three examples. 

\begin{ex}
Consider the tree in $\QQ^3$ with a finite edge running from $(0,0,0)$ to $(1,1,1)$ , infinite edges leaving $(1,1,1)$ in directions $(1,0,0)$ and $(0,1,1)$ and edges departing $(0,0,0)$ in directions $(0,-1,0)$ and $(-1,0,-1)$. Then $[0,t,1,t^{-1}]$ is isometric to $\Gamma$, with $0$, $t$,$ 1$ and $t^{-1}$ respectively corresponding to the endpoints of the above infinite rays. We have 
\begin{align*}
Z^{+,1} &= \{ 0 \} & Z^{+,2} &= \{ t \} & Z^{+,3} &= \{ t \} \\
Z^{-,1} &= \{ t^{-1} \} & Z^{-,2} &= \{ 1 \} & Z^{-,3} &= \{ t^{-1} \}
\end{align*}
Thus, the map $\phi$ is given by 
$$u \mapsto \left( \frac{u}{u-t^{-1}}, \frac{u-t}{u-1}, \frac{u-t}{u-t^{-1}} \right).$$
The image of this map is a genus $0$ curve $X$ with $\Trop X$ equal to the given tree.
\end{ex}

\begin{ex}
This time we choose a tree with no internal edges but complicated slopes. Consider the tree $T$ in $\QQ^3$ with no internal edges and four unbounded rays of slope $(1,2,3)$, $(5,-3,4)$, $(-7,1,-2)$, $(1,0,-5)$. Assuming that $\kk$ has characteristic $0$, the tree $[1,2,3,4] \subset \BT(K)$ is isometric to $T$. Our multisets $Z^{\pm}_i$ are
\begin{align*}
Z^{+,1} &= \{ 1, 2,2,2,2,2, 4 \} & Z^{+,2} &= \{ 1,1,3 \} & Z^{+,3} &= \{ 1,1,1,2,2,2,2 \} \\
Z^{-,1} &= \{ 3,3,3,3,3,3,3 \} & Z^{-,2} &= \{ 2,2,2 \} & Z^{-,3} &= \{ 3,3,4,4,4,4,4 \}
\end{align*}
For example, there are $5$ occurences of the number $4$ in $Z^{-,3}$ because ray number $4$ of our tree has slope $-5$ in the $x_3$ direction.

Our map $\phi$ is given by
$$u \mapsto \left(  \frac{(u-1)(u-2)^5 (u-4)}{(u-3)^7}, \frac{ (u-1)^2 (u-3)}{ (u-2)^3 }, \frac{ (u-1)^3 (u-2)^4}{ (u-3)^2 (u-4)^5 } \right).$$

Once again, the image of $\phi$ is a genus zero curve whose tropicalization is the given tree.
\end{ex}

\begin{ex}\label{Unfolded}
For our third example, we will engineer a genus zero zero tension curve such that $\iota$ is not injective. This curve is depicted in Figure~\ref{Genus0WithLoop} and will cause us several techinical difficulties in the Appendix. (In Figure~\ref{Genus0WithLoop}, we have indicated the slopes of the unbounded rays next to those rays. The plane $z=0$ is shown as the dashed parallelogram; the rays with slope $(0,0,-1)$ are below this plane while the rays with slope $(-2,1,1)$ and $(1,-2,1)$ lie above it.) Let $T$ be the tree shown in Figure~\ref{SCUnfolded}, where the unbounded rays are drawn with arrowheads and the internal edge has length $1$. We note that $T$ is isometric to $[ 0, 1, -1, t^{-1}, -t^{-1}, \infty ]$. So we take $X$ to be the curve $\PP^1(\KK) \setminus \{ 0, 1, -1, t^{-1}, -t^{-1}, \infty \}$. We now need to choose the map $\phi$. Writing $\phi$ in coordinates as $(\phi_1, \phi_2, \phi_3)$, we want $\phi_3$ to have a simple pole at $-1$ and $- t^{-1}$ and a simple zero at $1$ and $2 t^{-1}$. This way, the unbounded rays ending at $-1$ and $- t^{-1}$ will stretch downward with a slope whose third component is $-1$, the rays ending at $-1$ and $- 2 t^{-1}$ will stretch upward with a slope whose third component is $1$ and the rays ending at $0$ and $\infty$ will be contained in a plane where the third coordinate is constant. Similarly, we want $\phi_1$ to have a double zero at $-1$ and simple poles at $0$ and $- t^{-1}$ and we want $\phi_2$ to have a double zero at $- t^{-1}$ and simple pole and $-1$ and $\infty$. So we take $\phi$ to be given by
$$\phi : u \mapsto \left( \frac{(u+1)^2}{u(u+t^{-1})},  \frac{(u+t^{-1})^2}{(u+1)},  \frac{(u-1)(u-t^{-1})}{(u+1)(u+t^{-1})} \right).$$ 
For future reference, we note that $\Trop \phi(X)=\Trop \phi'(X)$ where
$$\phi' : u \mapsto \left( \frac{(u+1)^2}{u(u+t^{-1})},  \frac{(u+t^{-1})^2}{(u+1)},  \frac{(u-1)(u-2 t^{-1})}{(u+1)(u+t^{-1})} \right)$$

\begin{figure}
\centerline{\includegraphics{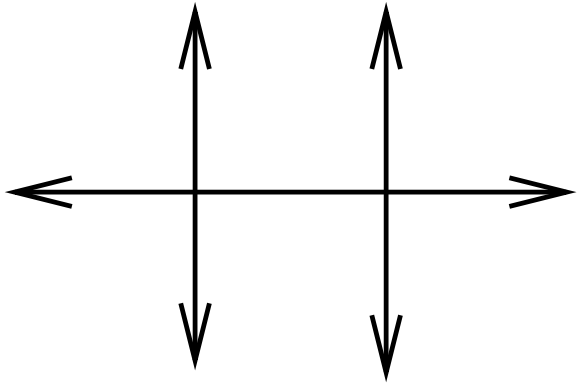}}
\caption{The tree $T$ in Example~\ref{Unfolded}} \label{SCUnfolded}
\end{figure}
\end{ex}

\begin{figure}
\centerline{\includegraphics{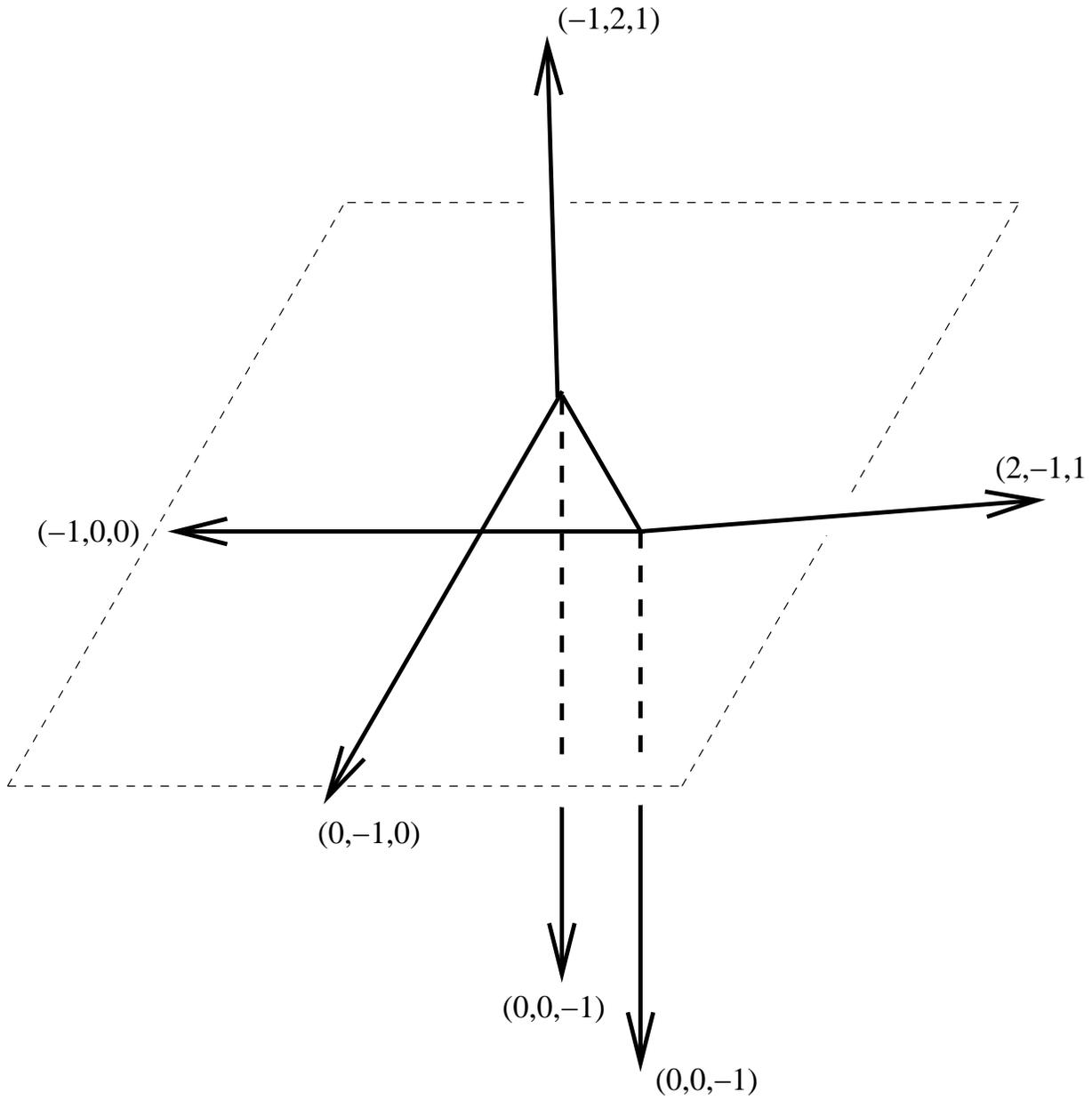}}
\caption{A tropicalization of a genus zero curve which crosses through itself} \label{Genus0WithLoop}
\end{figure}

\section{Tropical Curves of Genus One} \label{GenusOneSection}

Let $( \Gamma, \iota, m)$ be a zero tension curve where $\Gamma$ is connected with first Betti 
number $1$. This means that $\Gamma$ has a unique cycle; let $e_1$, \ldots, $e_r$ be the 
edges of this cycle and let $\sigma_i$ be $\sigma(e_i)$.

Our aim in this section is to prove 

\begin{TheoremGOE*} 
Let $(\iota, \Gamma, m)$ be a zero tension curve of genus one and degree $\delta$. Assume that the slopes of $\iota(e_1)$, \dots, $\iota(e_r)$ span $\QQ \otimes \Lambda$. Then there is a (punctured) genus one curve $X$ over $\KK$, and a map $\phi : X \to \TT(\KK, \Lambda)$ so that $\iota(\Gamma)=\Trop \phi(X)$.
\end{TheoremGOE*}

%Let us make the condition of the theorem more explicit:

%\begin{prop} \label{OrdinaryGenusOne}
%$(\Gamma, \iota,w)$ is ordinary if and only if the $\sigma_i$ span $\RR^n$.
%\end{prop}

%\begin{proof}
%Consider the equation $\sum \ell_i \sigma_i=0$ where $\ell_i \in \RR$. By definition, 
%$(\Gamma,, \iota, w)$ is ordinary if this equation has solution space of codimension $n$ in 
%$\RR^r$. In other words, the kernel of the map taking $\RR^r \to \RR^n$ \emph{via} $(\ell_1, \ldots, \ell_r) 
%\mapsto \sum \ell_i \sigma_i=0$ must be $(r-n)$-dimensional or, in other words, the map must 
%be surjective. This precisely says that the $\sigma_i$ span $\RR^n$.
%\end{proof}

We use Tate's nonarchimedean uniformizations of elliptic curves. A
good reference for this subject is sections 2 and 3 of \cite{Roq}. Let $q
\in \KK^*$ with $v(q)>0$. Tate constructs an elliptic curve $E$ over $\KK$
with a bijection $\pp$ from  $\KK^*/q^{\ZZ}$ to   $E(\KK)$. For $u$ and $z \in
\KK^*$, define
$$\Theta_z(u)=\prod_{j=-\infty}^0 (1-u/(z q^j)) \prod_{j=1}^{\infty} (1-q^j z/u).$$
For any $u$, $z$ and $q$ in $\KK^*$, there is a finite extension $K$ of $\KKK$ containing $u$, $z$ and $q$ and this product is convergent in the nonarchimedean topology on $K$ as long as none of the individual terms are zero. (We can not directly argue that the product is convergent in $\KK$ because $\KK$ is not complete.) Thus, the product above is well defined for
all $u \in \KK^* \setminus q^{\ZZ} \cdot (z^{+}, z^{-})$ and it satisfies
$\Theta_z(qu)=(-z/u) \Theta_z(u)$. (Remember that $\lim_{n \to \infty} q^n=0$ because $v(q)>0$.) Thus, if $Z^+=\{ z^+_1, \ldots,
z^+_k \}$ and $Z^-=\{ z^-_1, \ldots, z^-_k \}$ are finite multisubsets
of $K^*$ with the same cardinality and $\prod_{i=1}^k
(z_i^{+}/z_i^{-})=1$ then 
$$\phi_{Z^+, Z^-}(u):=\prod_{z \in Z^{+}_i} \Theta_z(u)/\prod_{z \in Z^{-}_i} \Theta_z(u)$$
 is a well defined function on
$(\KK^*/q^{\ZZ}) \setminus \left( \bigcup_{j=\infty}^{\infty} q^{j} \cdot \{ z_1^{+}, \ldots, z_k^{+},
z_1^{-}, \ldots, z_k^{-} \} \right)$. Consider this product as a function on
$E(\KK)$ with the appropriate points removed. Tate proves that this is a
meromorphic function with zeroes at the points $\pp(z_i^{+})$ and poles
at $\pp(z_i^{-})$. Every nonzero meromorphic function on $E(\KK)$, up to scalar multiples, occurs in this way. (See \cite[Section 2, Proposition 1]{Roq} for the statement that all the nonzero functions in the field which Roquette calls $M_K$ are of this form. See \cite[Section 2, Statement IV]{Roq} for the fact that this field is the meromorphic functions on an elliptic curve over $\KK$.)

% For $z^+$, $z^{-} \in
%K^*$, define
%$$\phi_{z^+, z^-}(u)=\prod_{j=-\infty}^0 \left(
%\frac{u/z^{+}-q^j}{u/z^{-}-q^j} \right) \prod_{j=1}^{\infty} \left(
%\frac{u-q^j z^+}{u-q^j z^{-}} \right)$$ 
%This product is convergent in the nonarchimedean topology on $K$  for
%all $u \in \KK^* \setminus q^{\ZZ} \cdot (z^{+}, z^{-})$ and it satisfies
%$\phi(qu)=(z^{+}/z^{-}) \phi(u)$. (Remember that $\lim_{n \to \infty} q^n=0$ because $v(q)>0$.) Thus, if $Z^+=\{ z^+_1, \ldots,
%z^+_k \}$ and $Z^-=\{ z^-_1, \ldots, z^-_k \}$ are finite multisubsets
%of $K^*$ with the same cardinality and $\prod_{i=1}^k
%(z_i^{+}/z_i^{-})=1$ then $\phi_{Z^+, Z^-}(u):=\prod_{i=1}^k
%\phi_{z_i^{+}, z_i^{-}}(u)$ is a well defined function on
%$(\KK^*/q^{\ZZ}) \setminus \left( \bigcup_{j=\infty}^{\infty} q^{j} \cdot \{ z_1^{+}, \ldots, z_k^{+},
%z_1^{-}, \ldots, z_k^{-} \} \right)$. Consider this product as a function on
%$E(\KK)$ with the appropriate points removed. Tate proves that this is a
%meromorphic function with zeroes at the points $p(z_i^{+})$ and poles
%at $p(z_i^{-})$. Every nonzero meromorphic function on $E(K)$, up to scalar multiples, occurs in this way. (See \cite{Roq}[Section 2, Proposition 1] for the statement that all the nonzero functions in the field Roquette calls $M_K$ are of this form and \cite{Roq}[Section 2, Statement IV] for the fact that this field is the meromorphic functions on an elliptic curve over $\KK$. Note that the only hypothesis that Roquette places on his field is that it be complete with respect to a nonarchimedean valuation.)

Our goal, given the input data $(\Gamma, \iota, m)$, is to construct a genus $1$ curve $X$ over $\KK$ and a rational map $\phi: X \to (\KK^*)^n$ such that $\Trop \phi(X)=\iota(\Gamma)$. The way our construction will proceed is to use $(\iota, \Gamma, m)$ to construct an element $q \in \KK^*$, with $v(q) >0$, and finite multisubsets $Z_1^+$, \ldots, $Z_n^+$, $Z_1^-$, \ldots, $Z_n^-$, of $\PP^1(\KK)$. We will have $\prod_{z \in Z_i^+} z = \prod_{z \in Z_i^-} z$, which will actually be the hardest part of the construction to achieve. Thus, we will have a rational map $\phi$ from the Tate curve $X:=\KK^*/q^{\ZZ}$ to $(\KK^*)^n$ by $\pp(u) \mapsto (\phi_{Z_1^+, Z_1^-}(u), \ldots, \phi_{Z_n^+, Z_n^-}(u))$, for $u \in (K^*/q^{\ZZ}) \setminus \left( \bigcup_{j=\infty}^{\infty} q^{j} \bigcup_{i=1}^n ( Z_i^+ \cup Z_i^- ) \right)$. We will have arranged our choices so that $\Trop \phi(X)=\iota(\Gamma)$. 

We begin by discussing the situation for an arbitrary choice of $q$ and $Z^{\pm}_1$, \ldots, $Z^{\pm}_n$. Later, we will specialize our discussion to the particular choices that will derive from the graph $\Gamma$. We will use $\gimel(q, Z^{\pm}_{\bullet})$ to denote the graph we will construct from $q$ and $Z^{\pm}_1$, \dots, $Z^{\pm}_n$ eventually $\gimel$ will be isomorphic to $\Gamma$. \footnote{The symbol $\gimel$ is the Hebrew letter ``gimmel" which makes the same sound as the Greek letter $\Gamma$. } We will drop the arguments of $\gimel$ when they should be free from context.

So, let $q \in \KK^*$ with $v(q) >0$ and let $Z_1^+$, $Z_1^-$, \dots, $Z_n^+$, $Z_n^-$ be $2n$ finite nonempty multisubsets of $\KK^*$ with $|Z_i^+|=|Z_i^-|$ for $i=1$, $2$, \dots $n$.  We introduce the notation $Z$ for $\bigcup_{i=\infty}^{\infty} q^i \left( \bigcup_{j=1}^n \left( Z^+_j \cup Z^-_j \right) \right)$. Let $\tilde{\gimel}(q,Z^{\pm}_{\bullet})=\bigcup_{z, z' \in Z} [z,z'] \subset \BT(\KK)$. The metric space $\tilde{\gimel}(q, Z^{\pm}_{\bullet})$ is invariant under multiplication by $q^{\ZZ}$, we define $\gimel(q, Z^{\pm}_{\bullet})=\tilde{\gimel}(q, Z^{\pm}_{\bullet})/q^{\ZZ}$.

\begin{Lemma}
The metric space $\tilde{\gimel}(q, Z^{\pm}_{\bullet})$ is an infinite tree. The quotient $\tilde{\gimel}(q, Z^{\pm}_{\bullet})/q^{\ZZ}$ is a finite graph, with first Betti number $1$, and $\gimel(q, Z^{\pm}_{\bullet})$ is isometric to a dense subset of $U$.
\end{Lemma}

Note that we call a graph ``finite" when it has a finite number of vertices and of edges. When the graph is equipped with a metric, as $\tilde{\gimel}$ is, we do \textbf{not} take ``finite" to imply that all of the edges have finite length.

\begin{proof}
First note that, for any $a$ and $b \in \KK^*$, and any $d \in [0, \infty]$ we either have $d \in [a, q^j b]$ for $j$ sufficiently large or we have $d \in [a, q^{-j} b]$ for $j$ sufficiently large. The same holds if we take $d \in [a, \infty]$ or $d \in [a, 0]$.  So $\tilde{\gimel}$, which by definition is $\bigcup_{z, z' \in Z} [z,z']$, is also $\bigcup_{z, z' \in Z} [0, \infty, z, z']$. Also, note that for any $a$ and $b \in \KK^*$, we have $[a,b, 0, \infty]
=[a,0, \infty] \cup [b,0, \infty]$. (Just check the three possible topologies for the tree  $[a,b, 0, \infty]$.) So $\tilde{\gimel}=\bigcup_{z \in Z} [0,z,\infty]$.

For $u \in \QQ$, set $T_u=\bigcup_{z \in Z,\ v(z)=u} [0,z,\infty]$, so $\tilde{\gimel}=\bigcup_{u \in \QQ} T_u$. If $u \neq u'$ then $T_u \cap T_{u'}=[0, \infty]$. So $\tilde{\gimel}$ is just a central path $[0,\infty]$ with the side stalk $T_u \setminus [0,\infty]$ stuck on for each $u \in v(Z)$. But, for every $u$, there are only finitely many $u \in Z$ with valuation $u$, so $T_u$ is just a finite tree. And $v(Z)$ is a union of finitely many copies of $v(q) \cdot \ZZ$, so in particular $v(Z)$ is a discrete subset of $A$. So we see that $\tilde{\gimel}$ really is just an infinite tree, consisting of an infinite path with a periodic pattern of finite trees branching off from it.

When we take the quotient of this infinite tree by the periodic shift by $v(q)$ along this path, we get a finite graph.
\end{proof}

In Figure~\ref{G&TG}, we show $\tilde{\gimel}$ on the left and $\gimel$ (on the right) for a sample choice of $q$ and $Z$. The action of $q$ is shown by the boldfaced arrow.

\begin{figure}
 \centerline{\scalebox{0.8}{\includegraphics{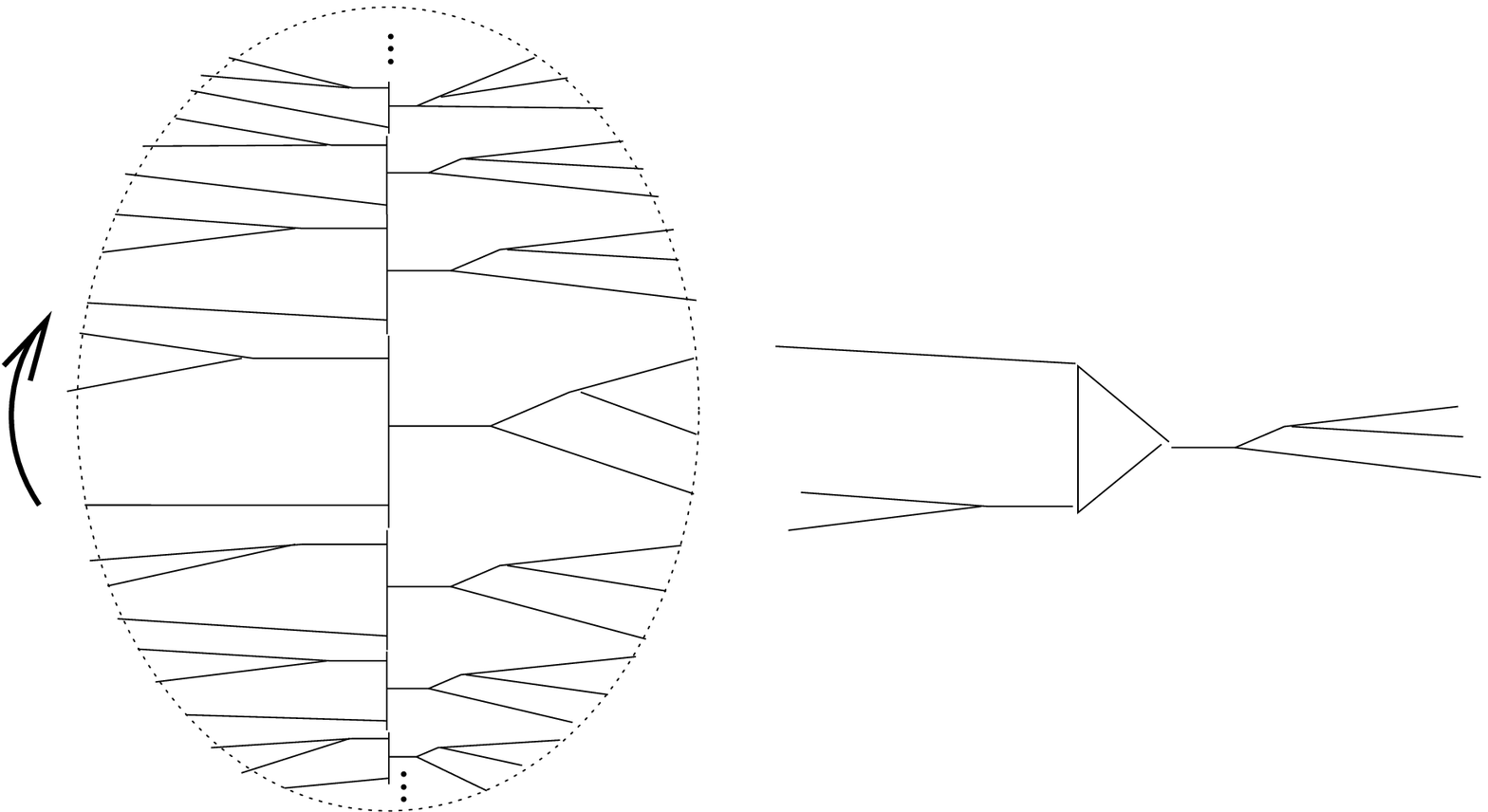}}}
\caption{The graphs $\tilde{\gimel}$ and $\gimel$} \label{G&TG}
\end{figure}

We now define a function $s$ assigning a vector in $\ZZ^n$ to each edge of $\gimel$; this function will eventually describe the slopes of edges of our tropical curve. Fix an index $i$ between $1$ and $n$ and an oriented edge $e$ of $\tilde{\gimel}$. Removing $e$ from $\tilde{\gimel}$ divides $\tilde{\gimel}$ into two components. This gives rise to partitions $Z_i^{\pm}=H_i^{\pm} \sqcup T_i^{\pm}$ of $Z_i^{\pm}$, where $H_i^{\pm}$ consists of those elements of $Z^{\pm}_i$ which are on the ``head" side of $e$ and $T_i^{\pm}$ consists of elements of $Z_i^{\pm}$ which are on the ``tail" side of $e$. Define
$$\tilde{s}_i(e)=|H_i^+| - |T_i^+| - |H_i^-| + |T_i^-|.$$
Note that each of these multisets is finite, so $\tilde{s}_i(e)$ is well defined. Also, note that $\tilde{s}_i(e)$ is zero for all but finitely many $e$. We now define
$$s_i(e)=\sum_{j=-\infty}^{\infty} \tilde{s}_i(q^j \cdot e).$$
Note that all but finitely many terms of this sum are zero. We then define $s(e)$ to be the vector $(s_1(e), \ldots, s_n(e))$. Note that $s(e)=s(qe)$, so we may think of $s$ a function on directed edges of $\gimel$, and note that reversing an edge negates $s$. 

\begin{lemma} \label{CombinatoricsGivesZTC}
The vectors $s(e)$ obey the zero tension condition.
\end{lemma}

\begin{proof}
We will show that the function $\tilde{s}$ on the edges of $\tilde{\gimel}$ obeys the zero tension condition, so the translates of $\tilde{s}$ will as well and hence their sum $s$ will. Let $v$ be an internal vertex of $\tilde{\gimel}$ and let $e_1$, $e_2$, \dots, $e_p$ be the edges incident to $v$, which we will direct away from $v$. Fix an index $i$ between $1$ and $n$, we need to show that $\sum_k \tilde{s}(e_k)_i =0$. Let $T_k$ be the component of $\tilde{\gimel} \setminus \{ v \}$ which contains  (the interior of) $e_k$. Let $D_k$ be the difference between the cardinalities of $Z^{+}_i \cap T_k$ and $Z^{-}_i \cap T_k$. Since $|Z^+_i|=|Z^-_i|$, we know that $\sum D_k=0$. We have $\tilde{s}(e_k)_i=D_k - \sum_{m \neq k} D_m$, so $\sum_{k=1}^p \tilde{s}(e_k)_i =(p-1) \sum_{k=1}^p D_k=0$.
\end{proof}

We will now consider the effect of adding the following additional condition:
\begin{equation}
\hbox{For $i=1$, \dots, $n$, \  } v\left( \prod_{z \in Z^{+}_i} z \right) = v\left( \prod_{z \in Z^{-}_i} z \right). 
\label{CombProductOne} \end{equation} 
Note that condition~(\ref{CombProductOne}) is an immediate consequence of the stronger condition
\begin{equation}
\hbox{For $i=1$, \dots, $n$, \ } \prod_{z \in Z^{+}_i} z  =  \prod_{z \in Z^{-}_i} z . \label{ProductOne}
\end{equation}

\begin{prop} \label{CycleCloses}
Suppose that condition (\ref{CombProductOne}) holds. Let $e_1$, $e_2$, \dots, $e_r$ be the edges of the unique cycle of $\gimel$, ordered and oriented cyclically. Let $\ell(e)$ be the length of the edge $e$ of $\tilde{\gimel}$. Then
$$\sum_{k=1}^r \ell(e_k) s(e_k)=0$$
\end{prop}

\begin{proof}
Pick an index $i$ between $1$ and $n$, we will show that $\sum_{k=1}^r \ell(e_k) s_i(e_k)=0$. Recall that we have $|Z_i^+| = |Z_i^-|$. Set $N=|Z_i^+|=|Z_i^-|$ and fix orderings $(z_1^{\pm}, z_2^{\pm}, \ldots, z_N^{\pm})$ of $Z_i^+$ and $Z_i^-$. For an oriented edge $e$ of $\tilde{\gimel}$ and an index $m$ between $1$ and $N$, define $\delta_m(e)$ to be $1$ if $z_m^+$ is on the head side of $e$ and $z_m^-$ is on the tail side of e;  define $\delta_m(e)$ to be $-1$ if $z_m^+$ is on the tail side of $e$ and $z_m^-$ is on the head side of e; define $\delta_m(e)$ to be $0$ if $z_m^+$ and $z_m^-$ are on the same side of $e$. Then $\tilde{s}_i(e)=\sum_{m=1}^N \delta_m(e)$. 

Now, we have
$$\sum_{k=1}^r \ell(e_k) s_i(e_k)=\sum_{k=1}^r \ell(e_k) \sum_{j= -\infty}^{\infty} q^j \tilde{s}_i(q^j e_k)=\sum_{e \subset [0,\infty]} \ell(e) \tilde{s}_i(e).$$
Using the expression in the previous paragraph for $\tilde{s}_i(e)$, we see that this is equal to
$$\sum_{e \subset [0,\infty]} \ell(e) \sum_{m=1}^N \delta_m(e)= \sum_{m=1}^N \sum_{e \subset [0,\infty]} \ell(e) \delta_m(e),$$
where we may interchange summation because all but finitely many terms are zero. Now, the inner sum 
$\sum_{e \subset [0,\infty]} \ell(e) \delta_m(e)$ is the (signed) sum of the lengths of all edges in $[0,\infty]$ which separate $z_m^+$ from $z_m^-$. In other words, the inner sum is the signed length of $[0, \infty] \cap [z_m^+, z_m^-]$. But, by Proposition~\ref{CrossRatio}, this is simply $v(z_m^+) - v(z_m^-)$. Plugging this into our sum, we see that the quantity we wish to show is zero is $\sum_{m=1}^N \left( v(z_m^+) - v(z_m^-) \right)$.  But, by condition~(\ref{CombProductOne}), we have $v\left( \prod_{z \in Z^{+}_i} z \right) - v\left( \prod_{z \in Z^{-}_i} z \right)=0$ so  $\sum_{m=1}^N \left( v(z_m^+) - v(z_m^-) \right)=0$.
\end{proof}

Define a continuous map $\tilde{f}_{q, Z^{\pm}_{\bullet}} : \tilde{\gimel} \to \QQ^n$ such that, if $e$ is any finite edge of $\tilde{\gimel}$, directed from vertex $x$ to vertex $y$ then $\tilde{f}_{q, Z^{\pm}_{\bullet}}(y)=\tilde{f}_{q, Z^{\pm}_{\bullet}}(x)+\ell(e) s(e)$ and $\tilde{f}_{q, Z^{\pm}_{\bullet}}(e)$ is the line segment connecting $\tilde{f}_{q, Z^{\pm}_{\bullet}}f(x)$ and $\tilde{f}_{q, Z^{\pm}_{\bullet}}(y)$. If $e$ is an infinite ray of $\gimel$ then $\tilde{f}_{q, Z^{\pm}_{\bullet}}(e)$ is an unbounded ray of slope $s(e)$. This map is unqiquely determined by the $s(e)$ up to translation by an element of $\QQ^n$. Using proposition~\ref{CycleCloses}, we see that $\tilde{f}$ factors through the quotient graph $\gimel$, we will write the map $\gimel \to \QQ^n$ as $f_{q, Z^{\pm}_{\bullet}}$. We now come to the fundamental computation that unites our combinatorial constructions with actual geometry:

\begin{prop} \label{GenusOneFundamentalComputation}
Given $q$ and $Z^{\pm}_1$, \dots, $Z^{\pm}_n$ subject to condition~(\ref{ProductOne}), define the curve $X$ and the map $\phi$ as above. Define also the graph $\gimel(q, Z^{\pm}_{\bullet})$ and the map $f_{q, Z^{\pm}_{\bullet}}$. (Since we have assumed (\ref{ProductOne}), we have (\ref{CombProductOne}), so we can define $f_{q, Z^{\pm}_{\bullet}}$.) Then $\Trop \phi(X)$ is a translation of $f_{q, Z^{\pm}_{\bullet}}(\gimel(q, Z^{\pm}_{\bullet}))$.
\end{prop}

This proof is very similar to the proof of Theorem~\ref{GenusZeroFundamentalComputation}. We will make the inconsequential abuse of notation of considering $\phi$ both as a map from (an open subset of) $X$ and from $\KK^* \setminus Z$. (Here $Z$, as above, is $\bigcup_{i=\infty}^{\infty} q^i \left( \bigcup_{j=1}^n \left( Z^+_j \cup Z^-_j \right) \right)$.) 

\begin{proof}
Let $u \in \KK^* \setminus Z$. Then $[Z \cup \{ u \}]$ is a tree, which is the union of $\tilde{\gimel}$ and a path ending at $u$. Let $b(u) \in \tilde{\gimel}$ be the other end of this path. We claim that, up to a translation,
$v(\phi(u))$ is $f(b(u))$. In other words, if $u_1$ and $u_2$ are distinct
members of $u \in \KK^* \setminus Z$, we must show that for each
$i$ between $1$ and $n$ we have
$$v(\phi_i(u_1))-v(\phi_i(u_2))=\iota(b(u_1))_i - \iota(b(u_2))_i.$$

It is enough to show this in the case where $b(u_1)$ and $b(u_2)$ lie in the same edge $e$ of 
$[Z]$. Let $s(e)=(s_1,\cdots s_n)$. We will fix one coordinate $i$ to pay attention to, so $i$ will not appear in our 
notation. Let $Z_i^{+}= \{ z_1^{+}, \ldots, z_N^{+} \}$ and $Z_i^{-}= \{ z_1^{-}, \ldots, 
z_N^{-} \}$.
%We may find constants $1 \leq s^+, s^{-} \leq n$ and order the $z_j^{\pm}$ such 
%that $z_j^{\pm}$ is on the $b(u_1)$ side of $e$ for $1 \leq j \leq s^{\pm}$ and on the 
%$b(u_2)$ side of $e$ for $s^{\pm} +1 \leq j \leq r$.
 Let $d$ be the distance from $b(u_1)$ to $b(u_2)$.

We have 
\begin{eqnarray*}
v(\phi_i(u_1))-v(\phi_i(u_2)) &=& v \left( \frac{\phi_i(u_1)}{\phi_i(u_2)} \right) \\
&=& v \left( \frac{ \prod_{m=1}^N \Theta_{z_m^{+}}(u_1)/ \Theta_{z_m^{-}}(u_1) }{\prod_{m=1}^N \Theta_{z_m^{+}}(u_2)/ \Theta_{z_m^{-}}(u_2) }\ \right) \\
&=& \sum_{m=1}^N v \left(  \frac{ \prod_{m=1}^N \Theta_{z_m^{+}}(u_1)/ \Theta_{z_m^{-}}(u_1) }{\prod_{m=1}^N \Theta_{z_m^{+}}(u_2)/ \Theta_{z_m^{-}}(u_2) } \right)
\end{eqnarray*}
Now, by definition,
\begin{eqnarray*}
 \frac{ \prod_{m=1}^N \Theta_{z_m^{+}}(u_1)/ \Theta_{z_m^{-}}(u_1) }{\prod_{m=1}^N \Theta_{z_m^{+}}(u_2)/ \Theta_{z_m^{-}}(u_2)} &=&
  \prod_{j=-\infty}^0 \left(
\frac{(1-u_1/(q^j z_m^{+})) (1-u_2/(z_m^{-} q^j))}{(1-u_1/(z_m^{-} q^j) )(1-u_2/(z_m^{+} q^j))} \right) \\
& & \quad \quad \times \prod_{j=1}^{\infty} \left(
\frac{(1-q^j z_m^+/u_1)(1-q^j z_m^-/u_2)}{(1-q^j z_m^{-}/u_1)(1-q^j z_m^+/u_2)} \right) \\
&=& \prod_{j=-\infty}^{\infty} c(u_1/ q^j, u_2 / q^j : z_m^-,  z_m^-).
\end{eqnarray*}
Here we may rearrange our product freely because in non-archimedean analysis all convergent sums and products converge absolutely.

So, by proposition~\ref{CrossRatio}, $v \left( \phi_{z_m^+, z_m^-}(u_1) /  \phi_{z_m^+, z_m^-}(u_2) \right)$ is the sum on $j$ (from $- \infty$ to $\infty$) of the signed length of $[u_1/ q^j,u_2 / q^j] \cap [z_m^+, z_m^-]$. Now, the sum on $m$ (from $1$ to $N$) of the signed length of $[u_1/ q^j,u_2 / q^j] \cap [z_m^+, z_m^-]$ is $d \tilde{s}_i(q^j e)$. So, interchanging summation, we obtain
$$v(\phi_i(u_1))-v(\phi_i(u_2)) = \sum_{j=- \infty}^{\infty} d \tilde{s}_i(q^{-j} e) = d s_i(e)=f(b(u_1))_i-f(b(u_2))_i$$
as desired.
\end{proof}

We now know how to built a genus one curve over $\KK$ with a rational map $\phi : X \to (\KK^*)^n$ such that $\Trop \phi(X)$ has a given form. In order to prove Theorem~\ref{GenusOneExists}, we need to use $(\iota, \Gamma, m)$ to find $q$ and to find multisubsets $Z^{\pm}_1$, \dots, $Z^{\pm}_n$ of $\KK^*$ such that 
\begin{itemize}
\item[(i)] for $i=1$, \dots, $n$, we have $|Z^{+}_i|=|Z^{-}_i|$
\item[(ii)] for $i=1$, \dots, $n$, we have $\prod_{z \in Z_i^{+}} z = \prod_{z \in Z_i^{-}} z$
\item[(iii)] the graph  $\gimel(q, Z^{\pm}_{\bullet})$ is isometric to $\Gamma$
\item[(iv)] the map $f_{q, Z^{\pm}_{\bullet}} $ corresponds to $\iota$ under the isometry $\gimel(q, Z^{\pm}_{\bullet}) \isomorph \Gamma$. 
\end{itemize}
We spend the rest of this section describing the necessary construction. We will first build sets $\Zo^{\pm}_1$, \dots, $\Zo^{\pm}_n$ which obey all of these conditions except that, instead of condition (ii) (which is the same as condition~(\ref{ProductOne})), they satisfy the weaker~(\ref{CombProductOne}). We will then perturb the $\Zo^{\pm}_i$ to give sets $Z^{\pm}_i$ satisfying this last condition.

Let $\ell$ be the length of the circuit of $\Gamma$. First, we choose $q$ so that $v(q)=\ell$. Let $\tilde{\Gamma}$ be the universal cover of $\Gamma$. The graph $\tilde{\Gamma}$ is an infinite tree which has one doubly infinite path with finite trees periodically branching off this path. Choose a point in the interior of an edge $e_0$ of the circuit of $\Gamma$ and cut $\Gamma$ at this point to produce a tree $T$, so $\tilde{\Gamma}$ is a union of a doubly infinite sequence of copies of $T$. The tree $T$ has two finite rays and a number of infinite rays. By Proposition~\ref{FindATree}, we can find a subset $U$ of $\BT(\KK) \cup \PP^1(\KK)$ such that $T \isomorph [U]$. There are two elements of $\BT(\KK)$ in $U$, which we will call $u_0$ and $u_1$; the rest of $U$ is contained in $\PP^1(\KK)$. The distance from $u_0$ to $u_1$ is $\ell=v(q)$, so we may (and do) use a transformation in $\PGL_2(\KK)$ to assume that $u_0=\Span_{\OO}((1,0), (0,1))$ and $u_1=\Span_{\OO}((q,0), (0,1))$. Since $u_0$ and $u_1$ are leaves of $[U]$, all the other elements of $U \setminus \{ u_0, u_1 \}$ must be contained in $v^{-1}( (0,\ell) ) \subset \KK^* \subset \PP^1(\KK)$. We set $\Zo =\bigcup_{j=- \infty}^{\infty} q^j (U \setminus \{ u_0, u_1 \})$. Then $[\Zo] \isomorph \tilde{\Gamma}$. 

We now describe how to choose the (finite) multisets $\Zo^{\pm}_1$, \ldots, $\Zo^{\pm}_n$ as subsets of the (infinite) set $\Zo$. Fix an index $i$ between $1$ and $n$. We first define multisubsets $Y_i^{+}$ and $Y_i^{-}$ of $U$. Let $v$ be a leaf of $\Gamma$, at the end of the edge $e$, and let $u$ be the element of $U$ at the end of the corresponding ray of $T$. Then $u$ will occur $|\sigma_v(e)_i|$ times in $Y^{\pm}_i$, where $\pm$ is the opposite of the sign of $\sigma_v(e)_i$, and will not occur at all in $Y^{\mp}_i$.  Note that the assumption that $(\iota, \Gamma, m)$ is zero tension forces $Y^+_i$ and $Y^-_i$ to have the same cardinality. We now modify $Y^{\pm}_i$ slightly to produce $\Zo^{\pm}_i$. Let $(s_1, \ldots, s_n)$ be the slope of the edge $e_0$ which we cut to make $T$. For each $i$ between $1$ and $n$ we take one element of $Y^+_i$ and multiply it by $q^{s_i}$, leaving the other elements of $Y^+_i$ unchanged. We call the resulting multiset $\Zo^+_i$. We take $\Zo^-_i=Y^-_i$. Then $\bigcup_{j=-\infty}^{\infty} q^j \bigcup_{i=1}^n (\Zo^+_i \cup \Zo^-_i)=\bigcup_{j=- \infty}^{\infty} q^j (U \setminus \{ u_0, u_1 \})$ so we have $\tilde{\gimel}(q,Y^{\pm}_{\bullet}) \isomorph \tilde{\gimel}(q,\Zo^{\pm}_{\bullet}) \isomorph \tilde{\Gamma}$ and $\gimel(q,Y^{\pm}_{\bullet}) \isomorph \gimel(q,\Zo^{\pm}_{\bullet}) \isomorph \Gamma$

We now need to check that, if $e$ is an oriented edge of $\Gamma$ then $\sigma(e)$ is equal to $s(e)$, where in the second expression we have thought of $e$ as an edge of $\gimel$. The $s(e)$ obey the zero tension condition, by Lemma~\ref{CombinatoricsGivesZTC}, so it is enough to check that $s(e)=\sigma(e)$ for (1) $e$ an unbounded ray and (2) $e=e_0$; the zero tension condition will then determine both $s(e)$ and $\sigma(e)$ for the remaining edges. If $e$ is an unbounded ray, with  the leaf $v$ at its end, then we forced there to be $\mp \sigma_v(e)_i$ elements of $\Zo^{\pm}_i$ lying at the end of the preimages of $e$ in $\tilde{\gimel}$. For each unbounded ray $e'$ of $\tilde{\gimel}$ (directed towards its leaf), $\tilde{s}(e')_i$ is $\pm$ times the number of elements of $\Zo^{\pm}_i$ at the end of $e'$.  So $s(e)_i$ is the sum of $\tilde{s}(e')_i$ over all preimages $e'$ of $e$ -- in other words precisely $\sigma(e)_i$ as desired. 

Now, let us consider the case $e=e_0$. If we computed $s(e_0)_i$ using $Y^+_i$ and $Y^-_i$, we would get zero as, for every preimage $e'$ of $e$ in $\tilde{\gimel}$, all of $U \setminus \{ u_0, u_1 \}$ lies to one side of $e'$. When we use $\Zo^+_i$ and $\Zo^-_i$ instead, only one element moves and that element moves past $\sigma(e_0)_i$ preimages of $e_0$, and  we have $\sigma(e_0)_i=s(e_0)_i$.

\begin{lemma} \label{CombProdOneHolds}
For each $i$ between $1$ and $n$, we have $v(\prod_{z \in \Zo_i+} z)=v(\prod_{z \in \Zo^-_i} z)$.
\end{lemma}

This proof is essentially reversing the proof of proposition~\ref{CycleCloses}.

\begin{proof}
Fix an index $i$ between $1$ and $n$. We choose orderings $( \zo_1^+, \zo_2^+, \ldots \zo_m^+)$ and  $( \zo_1^-, \zo_2^-, \ldots \zo_m^-)$ of $\Zo^+_i$ and $\Zo_i^-$; we want to establish that $\sum_{k=1}^m v(\zo^+_k/\zo^-_k)=0$. Now, $\zo^+_k/\zo^-_k$ is the crossratio $c(\zo_k^+, \zo_k^-: 0, \infty)$ so, by Proposition~\ref{CrossRatio}, $v(\zo^+_k/\zo^-_k)$ is the signed length of the intersection of $[0, \infty]$ with $[\zo_k^+, \zo_k^-]$. We can break up this intersection as a sum over the various edges in the doubly infinite path $[0, \infty]$, so $v(\zo^+_k/\zo^-_k)=\pm \sum_{e' \in [0, \infty] \cap [\zo_k^+, \zo_k^-]} \ell(e')$, where $\ell(e')$ is the length of $e'$ and the sign describes whether the orientation of the path $
[\zo_k^+, \zo_k^-]$ matches the orientation of $e'$ from $0$ to $\infty$. 

We know that, for any edge $e_j$ in the circuit of $\Gamma$, the signed number of intersections of the paths $[\zo_1^+, \zo_1^-]$, \dots, $[\zo^+_m, \zo^-_m]$ with the preimages of $e_j$ is $s(e_j)_i$. So $\sum_{k=1}^m v(\zo^+_k/\zo^-_k)=\sum_{j=1}^r \ell(e_j) s(e_j)_i$. But the $j^{\textrm{th}}$ summand on the right is the displacement in the $i^{\textrm{th}}$ coordinate between the two endpoints of $\iota(e_j)$. Since the edges $\iota(e_1)$, $\iota(e_2)$, \dots, $\iota(e_r)$ form a closed loop, the sum telescopes to zero.
\end{proof}

Now that we have proven Lemma~\ref{CombProdOneHolds}, we know that the map $f_{q, \Zo^{\pm}_{\bullet}}$ is well defined.

We have now shown that $\gimel \isomorph \Gamma$, and that, under this identification, $\sigma=s$. It follows that $f(\gimel)$ is a translation of $\iota(\Gamma)$, and we will no longer distinguish $\gimel$ from $\Gamma$. So, if we had $\prod_{z \in \Zo_i+} z=\prod_{z \in \Zo^-_i} z$, so that the map $\phi : X \to (\KK^*)^n$ existed in the first place, then $\Trop \phi(X)$, which we know would be a translation of $f(\Gamma)$, would also be a translation of $\iota(\Gamma)$. So we will have established Theorem~\ref{GenusOneExists} if we can simply show that $\prod_{z \in \Zo_i+} z=\prod_{z \in \Zo^-_i} z$. More precisely, what we show is that we can perturb each $\Zo^{\pm}_i$ to $Z^{\pm}_i$ so that this product relation holds without altering $\tilde{\gimel}$ as an abstract tree.

We now describe our perturbative method.  We remember that the edges of the circuit of $\Gamma$ are called $e_1$, \dots, $e_r$ and we introduce the notation $(v_{j-1}, v_j)$ for the endpoints of $e_j$, where the indices are cyclic modulo $r$.  Delete the interiors of the edges $e_1$, \dots, $e_r$ from $\Gamma$. Let $T_j$ be the component of the remaining forest containing the vertex $v_j$ of $\Gamma$.  We will choose constants $u_1$, \dots, $u_r$ in $\KK^*$ with $v(u_1)=\dots=v(u_r)=0$, which we think of as associated to the $T_j$. We modify the $\Zo^{\pm_i}$ as follows -- let $\zo$ be an element of $\Zo^{\pm}_i$.   Consider the component $T_j$ to which $\zo$ is attached. Then we multiply $\zo$ by $u_j$ to obtain a new element $z$. The multiset of thus modified elements will form the set $Z^{\pm}_i$. 

Now, multiplication by an element $u$ of $\KK^*$ with $v(u)=0$ is an automorphism of $\BT(\KK)$ which fixes (pointwise) the path $[0, \infty]$. This transformation modifies each component of $\tilde{\gimel} \setminus [0, \infty]$ by such an automorphism, so $\tilde{\gimel}$ is left unchanged as an abstract tree.  So we must understand how multiplication by $u_j$ effects the ratio $\prod_{z \in Z^{+}_i} z / \prod_{z \in Z^{-}_i} z$. All of the elements of $\Zo^{\pm}_i$ which are in $T_j$ are multiplied by $u_j$. Let $a_{ij}=|\Zo^{+}_i \cap T_j| - |\Zo^{-}_i \cap T_j|$. So our modification of the $\Zo^{\pm}$ multiplies $\prod_{z \in \Zo^{+}_i} z / \prod_{z \in \Zo^{-}_i} z$ by $\prod_{j=1}^r u_j^{a_{ij}}$. 

We want to know that we can choose $u_1$, \dots, $u_r$ in $v^{-1}(0)$ such that $ \prod_{j=1}^r u_j^{a_{ij}}=\left( \prod_{z \in \Zo^{+}_i} z / \prod_{z \in \Zo^{-}_i} z \right)$ for $i=1$, \dots $n$. We have shown (Lemma~\ref{CombProdOneHolds}) that the right hand side of this equation is in $v^{-1}(0)$. So we want to know that the map of abelian groups $v^{-1}(0)^r \mapsto v^{-1}(0)^n$ given by the matrix $A:=(a_{ij})$ is surjective. Since $\KK$ is algebraically closed, $v^{-1}(0)$ is a divisible group and hence it is enough to know that $A$ has rank $n$ over $\QQ$. We now turn to verifying this.

Let $\Gamma'$ be the graph obtained by taking the circuit of $\Gamma$ and adding an unbounded ray $r_j$ at each vertex $v_j$. Let $\iota'$ be the map from $\Gamma'$ into $\QQ^n$ where $\iota'$ restricted to the circuit of $\Gamma'$ is $\iota$ and the slope of $\iota'(r_j)$ is $(a_{1 j}, a_{2j} , \dots, a_{nj})$. Then $\Gamma'$ and $\iota'$ give a zero tension curve. By our assumption that $\Gamma$ is ordinary, we know that $\iota'(\Gamma)$ is not contained in any hyperplane. So, by Lemma~\ref{LocalMaxLemma}, the slopes of the $\iota'(r_j)$ span $\QQ^n$. Since the slopes of the  $\iota'(r_j)$ are the columns of $A$, this is the same as saying that $A$ has rank $n$.

So, we deduce that $A$ has full rank over $\QQ$ and therefore we can choose $u_1$, \dots, $u_r$ such that $\prod_{z \in \Zo^{+}_i} z / \prod_{z \in \Zo^{-}_i} z = \prod_{j=1}^r u_j^{a_{ij}}$ for $i=1$, \dots $n$. This, in turn allows us to construct the desired curve $X$ and desired map $\phi$. 

We conclude with an example of this construction. 

\begin{ex}
Consider the zero tension curve of genus one shown in Figure~\ref{GenusOneExample}, where the map 
$\iota$ is simply an injection. We take all of the edges of the square to have length $1$. Then the universal cover, $\tilde{\Gamma}$, is as shown in Figure~\ref{GenusOneCover}. We need to pick $q$ with $v(q)=4$; we make the simple choice $q=t^4$. We now cut the circuit of $\Gamma$ at the edge labeled $e$ to produce the tree $T$. Set $u_0=\Span_{\OO}( (1,0),(0,1))$ and $u_1=\Span_{\OO}( (1,0), (0,t^4))$. We take $U=\{ u_0, t^{0}, t^{1}, t^{2}, t^{3}, u_1 \}$. The reader is invited to check that, indeed, $[U]$ is isometric to $T$ and $[ \bigcup_{j=- \infty} ^{\infty} q^j (U \setminus \{u_0, u_1 \})]=[\{ t^k \}_{k \in \ZZ}] $ is isometric to $\tilde{\Gamma}$. The points $t^0$, $t^1$, $t^2$ and $t^3$ of $\PP^1(\KK)$ correspond to the rays of $\Gamma$ with slopes $(1,1)$, $(1,-1)$, $(-1,-1)$ and $(-1,1)$ respectively. 

\begin{figure}
 \centerline{\includegraphics{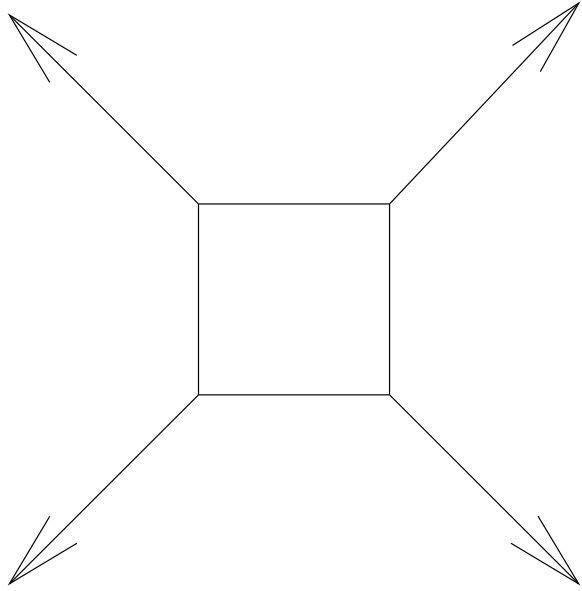}}
\caption{An example of a curve of genus one} \label{GenusOneExample}
\end{figure}

\begin{figure}
 \centerline{\includegraphics{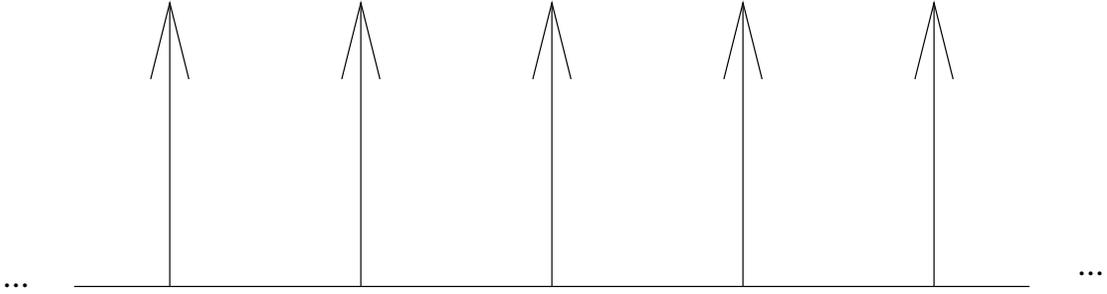}}
\caption{The universal cover of our example} \label{GenusOneCover}
\end{figure}

We now need to choose the multisets $Z^{+}_1$, $Z^{-}_1$, $Z^{+}_2$ and $Z^{-}_2$. First we pick subsets $Y^{\pm}_i$ of $U$:
$$Y^{+}_1 = \{ t^0, t^1 \} \quad Y^{-}_1 = \{ t^2, t^3 \} \quad Y^{+}_2 = \{ t^0, t^3 \} \quad Y^{-}_2 = \{ t^1, t^2 \}.$$
We then modify the $Y$'s to produce the $\Zo$'s. Specifically, we must multiply one of the members of $Y^{+}_1$ by $q$. We obtain
$$\Zo^{+}_1 = \{ t^4, t^1 \} \quad \Zo^{-}_1 = \{ t^2, t^3 \} \quad \Zo^{+}_2 = \{ t^0, t^3 \} \quad \Zo^{-}_2 = \{ t^1, t^2 \}.$$
This is the stage in the process where we would perturb the $\Zo$'s to produce the $Z$'s. However, we got lucky this time -- we already have $\prod_{z \in \Zo^{+}_i} z = \prod_{z \in \Zo^{-}_i} z$ for $i=1$, $2$, so no perturbation is necessary and we just take $Z^{\pm}_i=\Zo^{\pm}_i$. So we take the curve $X$ to be $\KK^*/t^4$ and we take the map $\phi$ to be given by $u \mapsto (\phi_{Z^{+}_1, Z^{-}_1}(u), \phi_{Z^{+}_2, Z^{-}_2}(u))$ or, to be completely explicit,
\begin{multline*}
u \mapsto \left( \prod_{j=- \infty}^{0}  \frac{(u/t^4-t^{4j})(u/t^1-t^{4j})}{(u/t^2-t^{4j})(u/t^3-t^{4j})} \prod_{j=1}^{\infty}  \frac{(u-t^{4j+4})(u-t^{4j+1})}{(u-t^{4j+2})(u-t^{4j+3})} , \right. \\
\left. \prod_{j=- \infty}^{0}  \frac{(u/t^0-t^{4j})(u/t^3-t^{4j})}{(u/t^1-t^{4j})(u/t^2-t^{4j})} \prod_{j=1}^{\infty}  \frac{(u-t^{4j})(u-t^{4j+3})}{(u-t^{4j+1})(u-t^{4j+2})} \right)
\end{multline*}

At this point, the reader may reasonably wonder how to extract an actual equation for the curve $X$. This is basically a problem of \emph{implicitization}, the recovery of the equation of an algebraic variety from a parametric representation, but it is worse than the usual implicitization problem because the parameterization is analytic, not algebraic. This problem deserves a great deal of study, which has been begun in \cite{Implicit1} and \cite{Implicit2}. We will describe here a straightforward, but unwieldly method. Let $\Sigma$ denote the complete fan whose rays point in directions $(1,1)$, $(1,-1)$, $(-1,-1)$ and $(-1,1)$. Then, by Proposition~\ref{ToricTransverse}, the closure of $\phi(X)$ in $\Toric(\Sigma)$ is torically transverse and has cohomology class corresponding to the linear relation $(1,1)+(1,-1)+(-1,-1)+(-1,1)=0$ between the rays of $\Sigma$. Thus,  $\phi(X)$ obeys an equation of the form $Ax+By+C+Dx^{-1}+Ey^{-1}=0$, for some $A$, $B$, $C$, $D$ and $E$ in $\KK$. Here we have written the coordinates on $(\KK^*)^2$ as $x$ and $y$. In order to find $A$, $B$, $C$, $D$ and $E$, expand the infinite products for $x$ and $y$ as Laurent series around $u-1$, and compare the coefficients of $(u-1)^k$ for $k=-1$, $0$, $1$, $2$, $3$. 
Note that this will express $A$, $B$, $C$, $D$ and $E$ in terms of infinite sums. There is no reason to expect a better result; our algorithm will usually produce curves defined by equations whose coefficients are not algebraic functions of $t$.

\end{ex}

\section{Superabundant Curves}
In this section, we prove Theorem~\ref{GenusOneSuperAbundant}, which we restate for the reader's convenience.

\begin{TheoremGOSA*}
Suppose that $\kk$ has characteristic zero. 

Let $(\iota, \Gamma, m)$ be a zero tension curve of genus one and degree $\delta$. Suppose also that $(\iota, \Gamma, m)$ is well spaced. Then there is a (punctured) genus one curve $X$ over $\KK$, and a map $\phi : X \to \TT(\KK, \Lambda)$ so that $\iota(\Gamma)=\Trop \phi(X)$.
\end{TheoremGOSA*}

See the end of Section~\ref{Results} for the definition of ``well spaced".

For each nonnegative rational number $R$, let $\Gamma_R$ be the subgraph of $\Gamma$ consisting of those points  with distance less than or equal to $R$ from the circuit of $\Gamma$. Let $H_R$ be the affine linear space spanned by $\Gamma_R$. Then $H_R$ is an increasing chain of affine spaces which increases at a finite number of discrete indices. Let $n_0$, $n_1$, \dots, $n_s=n$ be the dimensions of the various $H_R$'s. Let $R_{n_i}$ be the value of $R$ at which the jump from $n_i$ to $n_{i+1}$ occurs and let $\Gamma_{n_0}$, \dots, $\Gamma_{n_s}$ be the corresponding subgraphs of $\Gamma$. We will also occasionally need to refer to the open subset of $\Gamma_{n_j}$ where we do not include those points at distance precisely $R_{n_j}$ from the circuit of $\Gamma$; we denote this by $\mathring{\Gamma}_{n_j}$. Set $H_{n_i}=H_{R_{n_i}}$. We can make a change of basis in $\Lambda$ such that $H_{n_i}$ is $\Span_{\QQ}(e_1, e_2, \dots, e_{n_i})$. By the arguments in the previous section, we can find $Z^{\pm}_1$, \dots, $Z^{\pm}_n$, multisubsets of $\PP^1(\KK)$, and $q \in \KK^*$ with $v(q)$ equal to the length of the loop of $\Gamma$ such that $f(\gimel(q,Z^{\pm}_{\bullet}))=\iota(\Gamma)$. We can use the perturbation arguments of the preceding section to arrange that $\prod_{z \in Z^{+}_i} z=\prod_{z \in Z^{-}_i} z$ for $i=1$, \dots, $n_0$. However, for $i > n_0$, all we can conclude is that $v \left( \prod_{z \in Z^{+}_i} z \right)=v \left( \prod_{z \in Z^{-}_i} z \right)$. The strategy of our proof will be to show, by induction on $j$, that we can arrange for the equality $\prod_{z \in Z^{+}_i} z=\prod_{z \in Z^{-}_i} z$ to hold for $i \leq n_j$.

So, suppose that we have $Z^{\pm}_1$, \dots, $Z^{\pm}_n$ and $q$ such that $f(\gimel(q, Z^{\pm}_{\bullet}))=\iota(\Gamma)$ and such that $\prod_{z \in Z^{+}_i} z=\prod_{z \in Z^{-}_i} z$ holds for $i \leq n_j$. We introduce the notation $U$ for the group of units $u$ of $\OO$ such that $v(u-1) \geq R_{n_j}$.

\begin{Lemma} \label{Ratios}
For each $i$ between $1$ and $n$, the ratio $\prod_{z \in Z^{+}_i} z/ \prod_{z \in Z^{-}_i} z$ is in $U$.
\end{Lemma}

\begin{proof}
For $i \leq n_j$, we have $\prod_{z \in Z^{+}_i} z/ \prod_{z \in Z^{-}_i}=1$, which is in $U$. So fix $i > n_j$.  Let $T_1$, \dots, $T_r$ be the components of $\Gamma \setminus \mathring{\Gamma}_{n_j}$. We note that, for each $k$, we have $|Z^{+}_i \cap T_k| = |Z^{-}_i \cap T_k|$. This is because, by the zero tension condition,  $|Z^{+}_i \cap T_k| - |Z^{-}_i \cap T_k|$ is the $i^{\textrm{th}}$ component of the edge of $\Gamma_{n_j}$ pointing inward from $T_k$. Since this is an edge of $\Gamma_{n_j}$, that component is zero. So we can pair of the elements of  $Z^{+}_i \cap T_K$ with those of $Z^{-}_i \cap T_k$. In each pair $(z^+, z^{-})$, we have $v(z^{+}/z^{-}-1)=v(c(z^{+}, \infty; z^{-}, 0)-1)$ which is $\geq R$ by Proposition~\ref{CrossRatio}. Then $ \prod_{z \in Z^{+}_i} z/\prod_{z \in Z^{-}_i} z$ is a product of ratios which are in $U$ and hence is itself in $U$ as we claimed.
\end{proof}

We introduce the notation $w_i$ for $\prod_{z \in Z^{+}_i} z/ \prod_{z \in Z^{-}_i} z$, which we have just shown to be in $U$. When $1 \leq i \leq n_j$, then $w_i$ is simply $1$.

\begin{Lemma} \label{Divisible}
The abelian group $U$ is divisible.
\end{Lemma}

\begin{proof}
This is where we will use that $\kk$ has characteristic zero, and hence that $v(m)=0$ for every nonzero integer $m$. Let $u \in U$ and let $m$ be a nonzero integer. Then, since $\KK$ is algebraically closed, $u$ has an $m^{\textrm{th}}$ root in $\KK$, and even has a root which lies in $\OO$ and reduces to $1$ in $\kk$. (Proof -- let $u_1$, \dots, $u_m$ be the roots of $z^m -u$ in $\KK$. Since $mv(u_i)=v(u_i^m)=v(u)=0$, we know that all of the $u_i$ are in $\OO$. Let $\overline{u_i}$ be the reduction of $u_i$ in $\kk$. Then we have $z^m-1=\prod (z-\overline{u_i})$ in $\kk$ so, by unique factorization in $\kk[z]$, one of the $u_i$ reduces to $1$ in $\kk$.) So, let $(1+\pi)^m=u$, with $v(\pi) >0$. Then $u=1+m\pi + \pi^2 c$ where $c \in \OO$. As $v(m)=0$, we have $v(m \pi) < v(\pi^2 c)$ so $v(u-1)=v(m \pi)=v(\pi)$ and we deduce that $v(\pi) \geq R_{n_j}$ so $1+\pi$ is in $U$ as desired.
\end{proof}

Let $e_1$, \dots, $e_p$ be the edges of $\Gamma$ which are not in $\Gamma_{n_j}$, but which have endpoints that are in $\Gamma_{n_j}$. Let $s_1$, \dots, $s_p$ be the slopes of these edges, directed away from $\Gamma_{n_j}$. 

 \begin{Lemma} \label{PairedRays}
There exist scalars $a_1$, \dots, $a_p \in \QQ$ such that $\sum a_i s_i =0$ and such that, if $e_c$ and $e_d$ are distinct edges with the same endpoint, then $a_c \neq a_d$.
\end{Lemma}

\begin{proof}
Let $\overline{s_i}$ be the image of $s_i$ in $\QQ^n / H_{n_j}$. We first show that we can find $a'_1$, \dots, $a'_p$ obeying the required inequalities with $\sum a'_i \overline{s'_i}=0$. Let $L \subset \QQ^p$ be the space of $p$-tuples of rational numbers $(a'_1, \ldots, a'_p)$ such that $\sum a'_i \overline{s_i} =0$. We want to show that $L$ has a point not contained in any of the hyperplanes $a'_c=a'_d$, where $(c,d)$ ranges over pairs of indices such that $e_c$ and $e_d$ share an endpoint. Since $L$ is a linear space, we just need to show that $L$ is not contained in any of these hyperplanes.

Let $e_c$ and $e_d$ share the endpoint $x$, which is necessarily a boundary vertex of $\Gamma_{n_j}$. Suppose, for the sake of contradiction that $L$ is contained in the hyperplane $a_c=a_d$. Then in particular, $L$ does not contain any point with $a'_c=0$ and $a'_d=1$. This means that $\overline{s_d}$ is not in the span of $\{ \overline{s_i} \}_{i \neq c,d }$

Equivalently, $s_d$ is not in $V := \Span(H_{n_j} \cup \{ s_i \}_{i \neq c,d})$. Let $H$ be a hyperplane in $\QQ^n$ which contains $V$ and does not contain $s_{d}$. Then $x$ is at distance $R_{n_j}$ from the loop of $\Gamma$, while every other boundary vertex of $\Gamma \cap H$ is farther away, contradicting our hypothesis that $\Gamma$ is well spaced. 

So, now we have $(a'_1, \ldots, a'_p)$ obeying the required inequalities with $\sum a'_i s_i$ in $H_{n_j}$. Let $x_1$, \dots, $x_r$ be the boundary vertices of $\Gamma_{n_j}$ and let $t_k$ be the slope of the edge of $\Gamma_{n_j}$ pointing inward from $x_k$. Then, by Lemma~\ref{LocalMaxLemma}, the $t_k$ span $H_{n_j}$. Let $\sum a'_i s_i =\sum b_k t_k$. Now, by the zero tension condition, $t_k=-\sum_{e_i \ni x_k} s_i$. So we have $\sum_{i=1}^p a'_i s_i + \sum_{k=1}^r  b_k \sum_{x_k \ni e_i} s_i=0$. We regroup this expression and take the coefficient of $s_i$ to be our $a_i$. If $e_c$ and $e_d$ share the endpoint $x_k$ then the coefficients of $s_c$ and $s_d$ in this expression are $a'_c+b_k$ and $a'_d+b_k$. Since $a'_c \neq a'_d$, we also have $a'_c+b_k \neq a'_d+b_k$ and we have achieved the goal.

\end{proof}

Our strategy will be to choose $u_1$, \dots $u_p \in U$ and, for each $Z^{\pm}_i$, multiply those elements of $Z^{\pm}_i$ which lie in the component of $\Gamma \setminus \Gamma_{n_j}$ containing $e_k$ by $u_k$. This will have the effect of multiplying $w_i$ by $\prod_{k=1}^p u_k^{s^i_k}$ where $s_k=(s_k^1, \dots, s_k^n)$. We need to achieve two things -- we must make the new values of $w_i$ be $1$ for $i \leq n_{j+1}$ and we must not change the fact that $f(\gimel)=\iota(\Gamma)$.

Now, by lemma~\ref{LocalMaxLemma} applied to $\Gamma_{R_{n_j}+\epsilon}$ for some small $\epsilon$, we know that the $s_j$ span $H_{n_{j+1}}$. Since $U$ is a divisible group, this means that we can arrange to make $\prod_{k=1}^p u_k^{s^i_k}$ assume any value that we want for $1 \leq i \leq n_{j+1}$ and in particular, we can make $w_1=w_2=\cdots=w_{n_{j+1}}=1$. The trouble is that we might no longer have $f(\gimel)=\iota(\Gamma)$. Now, multiplication by an element of $u$ will fix all of the points of $\Gamma_{n_j}$ and will move each component of $\Gamma \setminus \Gamma_{n_j}$ by an isomorphism. Our only fear is that two of these components which are connected to $\Gamma_{n_j}$ at the same vertex, say $x$, will swing so that the edges at which they attach to $x$, say $e_c$ and $e_d$, overlap for some length. 

Let $(a_1, \dots, a_q) \in \QQ^q$ be the vector found in Lemma~\ref{PairedRays}. Assume that we have arranged for $w_i$ to be $1$ for $i=1$, $2$, \dots, $n_{j+1}$. Now consider choosing some $u \in U$ and, for each $Z^{\pm}_i$, further multiplying those elements of $Z^{\pm}_i$ which lie in the component of $\Gamma \setminus \Gamma_{n_j}$ containing $e_k$ by $u^{a_k}$. This will multiply $w_i$ by $u^{\sum a_k s_k^i}=u^0=1$, so this will not break our achievement of making the $w_i=1$ (for $i=1$, \dots, $n_{j+1}$.) Also, if $e_c$ and $e_d$ share an end point $x$, then the components of $\Gamma \setminus \Gamma_{n_j}$ containing $e_c$ and $e_d$ will be multiplied by $u^{a_c}$ and $u^{a_d}$ respectively, two different scalars, and so, for generic $u \in U$, they will not overlap.

We have now completed the inductive step, showing how to make $w_i$ equal to $1$ for $i \leq n_{j+1}$. Continuing in this manner, we will eventually have all the $w_i$ equal to $1$, and thus the map $\phi$ will be well defined. At that point, we will have a curve $X$ and a map $\phi$ with $\Trop \phi(X)=\iota(\Gamma)$ as desired.

\section{The necessity of well-spacedness and the $j$-invariant}

We have been studying genus $1$ curves over $X$ by identifying them with Tate curves. It is therefore natural to ask to what extent we may assume that a genus one curve over $\KK$ is a Tate curve.

\begin{Theorem}
Let $X$ be a genus one curve over $\KK$ and $\overline{X}$ its projective completion. Let $j$ be the $j$-invariant of $X$. Let $\mathcal{X}$ be the extension of $(\overline{X}, \overline{X} \setminus X)$ to a (flat, projective) family of stable curves with marked points over $\OO$. (There is no need to take a branched cover of $\OO$ because $\KK$ is algebraically closed.)\footnote{More precisely, after taking a finite base extension of $\KKK$, we may assume that $X$ is defined over $\KKK$. Since $\OOO$ is noetherian, after another finite base extension, we may complete $X$ to a family of stable curves over $\OOO$. Then go back to $\OO$, which was all we ever actually cared about.}  Then the following are equivalent:

\begin{enumerate}
\item $X$ can be expressed as a Tate curve for some $q \in \KK^*$ with $v(q)>0$.
\item $v(j) <0$.
\item The fiber of $\mathcal{X}$ over $\Spec \kk$ is a union of genus zero curves.
\end{enumerate}

Moreover, if these equivalent conditions hold, we have $v(q)=-v(j)$.
\end{Theorem}

\begin{proof}
The equivalence of (1) and (2) is \cite[Section 3, Statement VIIIa]{Roq}. (Observe that Roquette's conditions (i) and (iii) are automatic when $\KK$ is algebraically closed, the latter because the Hasse invariant lives in the group $\KK^*/(\KK^*)^2$ which is trivial for $\KK$ algebraically closed.) If $N$ is the number of points in $\overline{X} \setminus X$, then $\mathcal{X}$ gives a map from $\Spec(\OO)$ to the moduli space of stable genus 1 curves with $N$ marked points. Forgetting all but one of the points, we get a map to the moduli space of stable genus 1 curves with one marked point. This moduli space is a copy of the projective line, commonly known as the $j$-line. Condition (3) is equivalent to saying that this map sends $\Spec \kk$ to the point at infinity on the $j$-line, which is exactly what condition (2) says. We have $v(j)=-v(q)$ because $j=q^{-1}+\sum_{i=0}^{\infty} c_i q^i$ where each $c_i$ is an integer and hence has nonnegative valuation.
\end{proof}

Combining the above with Corollary~\ref{TotRat}, we see that if $(\phi,X)$ is trivalent and $\Trop \phi(X)$ has first Betti number one then $X$ is a Tate curve. In this case, we can show that the condition that $\Trop \phi(X)$ be well spaced is necessary, and that the length of the loop of $\Trop \phi(X)$ is necessarily $-v(j)$. Recall that \emph{every} nonzero meromorphic function on a Tate curve is of the form $a \phi_{Z^{+}, Z^{-}}$ for some $Z^{+}=\{ z^{+}_1, \dots, z^{+}_n \}$ and $Z^{-}=\{ z^{-}_1, \dots, z^{-}_n \} \subset \KK^*$ with $\prod z^{+}_i=\prod z^{-}_i$. 

\begin{prop} \label{Necessity}
Let $(\phi,X)$ be a trivalent tropical curve of genus $1$ and $(\iota, \Gamma,m)$ a zero tension curve parameterizing $\Trop \phi(X)$, with $b_1(\Gamma)=1$ and $\iota$ injective on $\Gamma$. Then the length of the loop of $\Gamma$ is $-v(j(X))$, where $j(X)$ is the $j$-invariant of $X$. Now suppose that $\kk$ has characteristic zero. Then $(\iota, \Gamma, m)$ is well spaced.
\end{prop}

The idea that the length of the circuit of $\Gamma$ is the tropical analogue of the $j$-invariant was suggested in \cite{Mikh2} and pursued in \cite{jInv}. E. Katz and H. and T. Markwig, have proven that the length of the circuit of $\Gamma$ is the $j$-invariant of a trivalent curve for tropicalizations of cubic curves in $\PP^2$.

\begin{proof}
Our hypotheses are enough to ensure that $X$ is a Tate curve $\KK^*/q^{\ZZ}$ for some $q$ and that the map $\phi : X \to (\KK^*)^n$ is given by $(\phi_{Z^{+}_1, Z^{-}_1}, \ldots, \phi_{Z^{+}_n, Z^{-}_n})$ for some $Z^{\pm}_1$, \dots, $Z^{\pm}_n$. 

By our combinatorial construction of $\gimel(q, Z^{\pm}_{\bullet})$, the length of the loop of $\gimel$ is $v(q)=-v(j(X))$. Lemma~\ref{trivalent} ensures that $f_{q, Z^{\pm}_{\bullet}}$ is injective and thus the loop of $\gimel$ and the loop of $\Gamma$ have the same length.

Now, assume that $\kk$ has characteristic zero and (for the sake of contradiction) that $(\iota, \Gamma,m)$ is not well-spaced. After a change of coordinates, we may assume that $H$ is the $n^{\textrm{th}}$ coordinate hyperplane. Assume for the sake of contradiction that $x$ is the boundary vertex of $\Gamma \cap H$ which is closest to the loop of $\Gamma$ and let $R$ denote the distance from $x$ to the loop of $\Gamma$. Then $x$ is trivalent by hypothesis. One of the edges of $x$ is between $x$ and the loop of $\Gamma$, this edge is contained in $H$ so its $n^{\textrm{th}}$ component is $0$. Let the $n^{\textrm{th}}$ components of the other two edges be $a$ and $-a$ for some positive integer $a$. Let $\{ z_1^{+}, \dots, z_a^{+} \}$ and $\{ z_1^{-}, \dots, z_a^{-} \}$ be the elements of $Z^{+}_n$ and $Z^{-}_n$ respectively which are beyond these edges. Then, we have $v(z^{+}_i/z^{+}_j-1) > R$ for any indices $i$ and $j$ from $1$ to $a$ and similarly $v(z^{-}_i/z^{-}_j-1) > R$. On the other hand, $v(z^{+}_i/z^{-}_j-1)=R$. So, 
$$\prod_{i=1}^a (z^{+}_i/z^{-}_i)= (z^{+}_1/z^{-}_1)^a (1+t^R c)$$
for some $c$ with $v(c)>0$ and thus $v(\prod_{i=1}^a (z^{+}_i/z^{-}_i)-1)=R$. We have used that $\kk$ has characteristic zero to deduce that, if $v(u-1)>0$ then $v(u^a-1)=v(u-1)$.

Pair off the elements of $Z^{+}_n \setminus  \{ z_1^{+}, \dots, z_a^{+} \}$ with elements of  $Z^{-}_n \setminus \{ z_1^{-}, \dots, z_a^{-} \}$ that lie in the same component of $\Gamma \setminus (H \cap \Gamma)$, this is possible by the same argument as in the proof of Lemma~\ref{Ratios}. We write the pairs as $(z^{+}_i, z^{-}_i)$ for $i > a$. Then in each pair $(z^{+}, z^{-})$, we have $v(z^{+}/z^{-}-1) > R$. But we are supposed to have $\prod_{z \in Z^{+}_n} z/\prod_{z \in Z^{-}_n} z=1$, so $\prod_{i=1}^a (z^{+}_i/z^{-}_i)=\left( \prod_{i>a} (z^{+}_i/z^{-}_i) \right)^{-1}$. Then the left hand side differs from $1$ by an element of valuation $R$ while the right hand side differs from $1$ by an element of valuation greater than $R$, a contradiction.
\end{proof}

\section{Appendix: Tropical Generalities and Theorem~\ref{NecessaryConditions}}

In this appendix, we describe some general tropical constructions. Our eventual aim is to explain and prove Theorem~\ref{NecessaryConditions} in the case where $X$ is a curve such that $\In_w X$ is reduced for every vertex $w$ of $\Trop X$. In this section we will prove many theorems for tropical varieties in general and then apply them to the case of curves. For clarity, we use $X$ for a variety which may have any dimension and $C$ for a curve. We restate our goal for the reader's convenience.

\begin{TheoremNC}
Let $C$ be a connected (punctured) curve of genus $g$ over $\KK$ equipped with a map $\phi : C \to \TT(\KK, \Lambda)$. Let $\delta \subset \Lambda$ be the degree of $(C,\phi)$. Suppose that $\In_w X$ is reduced for every vertex $w$ of $\Trop X$. Then there is a connected zero tension curve $(\iota, \Gamma,m)$ of degree $\delta$ and genus at most $g$ with $\iota(\Gamma)=\Trop \phi(C)$.

If there is some vertex $w$ of $X$ for which $\In_w X$ is not rational, then we can take $b_1(\Gamma) < g$.
\end{TheoremNC}

We begin by discussing the local geometry of $\Trop X$.

\begin{prop}
Let $X \subset \TT(\KK, \Lambda)$ and let $w \in \Trop X$. Then the star of $\Trop X$ at $w$ is $\Trop \In_w X$.
\end{prop}

\begin{proof}
By definition, if $\Delta$ is a polyhedral complex and $w \in \Delta$, then the star of $\Delta$ at $w$ is the set of $v \in \QQ^n$ such that $w + \epsilon v \in \Delta$ for $\epsilon$ sufficiently small. Thus, this result follows from Proposition~\ref{PerturbInit} and description (3) of $\Trop X$ in Theorem~\ref{DefnsOfTrop}.
\end{proof}

The next result roughly states that, when $w$ is a pretty general point of $\Trop X$ then $\In_w X$ is very symmetric. For $\sigma$ any polytope in $\QQ \otimes \lambda$, let $H(\sigma)$ denote the linear space spanned by vectors of the form $w-w'$ for $w$, $w' \in \sigma$. If $\sigma$ is defined by inequalities of the form $\{ w : \langle \lambda, w \rangle \geq a \}$ for various $\lambda \in \Lambda^{\vee}$, then there is a subtorus $\exp(H(\sigma)) \subseteq \TT(\kk, \Lambda)$ which has dimension $\dim \sigma$. The torus $\exp(H(\sigma))$ is cut out by the equations $\chi^{\lambda}=1$ where $\lambda$ ranges over those elements of $\Lambda^{\vee}$ such that $\langle \lambda, \cdot \rangle$ is constant on $\sigma$. 

\begin{prop}
Let $\sigma$ be a face of a good subdivision of $\Trop X$. Then, for $w$ in the relative interior of $\sigma$, the variety $\In_w X$ is invariant under the action on $\TT(\kk, \Lambda)$ of the torus $\exp(H(\sigma))$.\end{prop}

\begin{proof}
By proposition \ref{PerturbInit}, for $u \in H(\sigma)$, we have $\In_u \In_w X=\In_w X$. But $\In_u \In_w X$ is clearly invariant under $\exp(\QQ \cdot u)$, as, for any $f \in \kk[\Lambda^{\vee}]$, the polynomial $\In_u f$ is homogenous with respect to the grading by $u$. 
\end{proof}

Note that this has an immediate corollary:

\begin{cor} \label{DimAtMost}
The dimension of $\Trop X$ is less than or equal to that of $X$.
\end{cor}

\begin{proof}
Let $\dim X=d$ and suppose for the sake of contradiction that $\dim \Trop X > d$. Then there must exist a face $\sigma$ of $\Trop X$ with dimension $>d$. Let $w$ lie in the relative interior of $\sigma$. By description (5) of $\Trop X$ in proposition~\ref{DefnsOfTrop}, $\In_w X$ is not empty. Thus, since $\exp(H(\sigma))$ acts freely on $\TT(\KK, \Lambda)$, we have $\dim \In_w X \geq \dim \exp(H(\sigma))=\dim \sigma > d$. But $\In_w X$ is a flat degeneration of $X$, so its dimension can be no more than $d$, a contradiction.
\end{proof}

Let $\sigma$ be a face of a good subdivision of $\Trop X$ and let $w$ lie in the relative interior of $\sigma$. Set  $Y_{\sigma}: = \In_w X/\exp(H(\sigma)) \subset \TT(\kk, \Lambda/(\Lambda \cap H(\sigma)) )$. Then $\TT(\kk, \Lambda)$ splits (noncanonically) as $\exp(H(\sigma)) \times  \TT(\kk, \Lambda/(\Lambda \cap H(\sigma)) )$ and we have $\In_w X = \exp(H(\sigma)) \times Y_{\sigma}$ and $\Trop \In_w X=H(\sigma) \times \Trop Y_{\sigma}$. We have $\dim \In_w X=\dim \sigma+\dim Y_{\sigma}$. When $X$ is pure of dimension $d$, this allows us to improve Corollary~\ref{DimAtMost}. 

\begin{prop} 
If $X$ is integral and $d$-dimensional then, for any $w \in \Trop X$, $\In_w X$ is pure of dimension $d$. As a corollary, if $\sigma$ is any face of a good subdivision of $\Trop X$ then $d=\dim \sigma+\dim Y_{\sigma}$.
\end{prop}

Note that ``pure of dimension $d$" means that every irreducible component of a scheme has dimension $d$, but it permits that there may be associated primes of smaller dimension.

\begin{proof}
After taking a finite extension of $\KKK$, we may assume that $t^w \in \KKK$ and $X=\boldsymbol{X} \times \KK$ for some $\boldsymbol{X} \subset \TT(\KKK,\Lambda)$.  Recall our description of $\In_w X$ as the fiber over $\Spec \kk$ of $\overline{t^{-w} \cdot \boldsymbol{X}}$. Now, $\overline{t^{-w} \cdot \boldsymbol{X}}$ is integral because it is the Zariski closure of an integral variety. (A subring of a domain is a domain.) Taking the closure increases the dimension by one. Taking the quotient by $t$ brings the dimension back down to $d$ again, as in a domain any single non-zero element is a regular sequence. So the fiber of $\overline{t^{-w} X}$ over $\Spec \kk$ is $d$-dimensional again. 
\end{proof}

In particular, suppose that $X$ is pure of dimension $d$ and $\sigma$ is a $d$-dimensional face of a good subdivision of $\Trop X$. Then $Y_{\sigma}$ is $0$-dimensional and thus has finite length. Let $m_{\sigma}$ be the length of $Y_{\sigma}$. Note that we now know, if $C$ is a curve in $\TT(\KK, \Lambda)$, that $\Trop C$ is pure of dimension $1$. In other words, $\Trop C$ is a graph in this case.

We now prove a result which shows, in particular, that if $C$ is a curve then $\Trop C$ is a zero tension curve.

Let $\Sigma \subset \QQ \otimes \Lambda$ be a finite rational polyhedral complex, pure of dimension $d$, with a positive integer $m_{\sigma}$ assigned to each maximal face of $\Sigma$. Let $\rho$ be a $(d-1)$-dimensional face of $\Sigma$ and let $\sigma_1$, \dots, $\sigma_k$ be the collection of $d$-dimensional faces of $\Sigma$ which contain $\rho$. Then $\Lambda / (H(\rho) \cap \Lambda)$ is a lattice of dimension $(n-d+1)$ and the image of $H(\sigma_i) \cap \Lambda$ is a ray in this lattice. Let $v_i$ be the minimal lattice vector along this ray. We say that the \textbf{zero tension condition holds at $\rho$} if $\sum m_{\sigma_i} v_i=0$. We say that $\Sigma$ is a \textbf{zero tension complex} if the zero tension condition holds at every $(d-1)$-dimensional face $\rho$ of $\Sigma$. Note that, if $\Sigma'$ is a subdivision of $\Sigma$ and we define $m_{\sigma'}$ to be $m_{\sigma}$ where $\sigma'$ is a $d$-dimensional face of $\Sigma'$ and $\sigma$ is the face of $\Sigma$ containing $\sigma'$, then $\Sigma'$ is a zero tension complex if and only if $\Sigma$ is. We will now show that tropical varieties are zero tension complexes. We first tackle the case of a curve with constant coefficients.

\begin{prop} \label{ZeroTensionForCCCurves}
 Let $C \subset \TT(\kk, \Lambda)$ be an algebraic curve. Let $\{ \sigma_1, \ldots, \sigma_N \}$ be the degree of $C$; write $\sigma_i$ as $d_i \rho_i$ where $d_i$ is a positive integer and $\rho_i$ a primitive lattice vector of $\Lambda$. Then the rays of $\Trop C$ are precisely in the directions $\rho_1$, \dots, $\rho_N$ and the sum of the multiplicities of the rays in direction $d_i$ is $\rho_i$. As a corollary, $\sum \sigma_i=0$.
\end{prop}

\begin{proof}
Let $\tilde{C}$ be the projective completion of the normalization of $C$ and $\phi : \tilde{C} \dashedto \TT(\kk, \Lambda)$ the rational map with image $C$. Let $w \in \Lambda$ be a primitive lattice vector. We will show that, first, $w \in \Trop C$ if and only if there there is a multiple of $w$ in the degree of $C$ and, second, that $w$ has the same multiplicity in both cases. 

Without loss of generality, we can choose coordinates such that $\Lambda \isomorph \ZZ^n$ and $w=e_n$. Then we can embed $\TT(\kk, \Lambda)$ into $Y := (\kk^*)^{n-1} \times \kk$. Let $C'$ be the subset of $\tilde{C}$ on which $\phi$ can be extended to a map to $Y$ and write $\phi'$ for the extended map. The closure of $\phi(C)$ in $Y$ is $\phi'(C')$. We claim that we have the equality 
\begin{equation}
\In_w C = \left( \phi'(C') \cap \{ x_n=0 \} \right) \times \kk^* \label{SchemeInitEqual} \end{equation}
of subschemes of $(\kk^*)^n$.
To see this, let $f$ be a polynomial in $\kk[\Lambda^{\vee}]=\kk[x_1^{\pm}, \ldots, x_{n-1}^{\pm}, x_n^{\pm}]$.  We must show that $f$ vanishes on the left hand scheme if and only if it vanishes on the right hand scheme. Note that both sides are clearly homogenous in the $n^{\textrm{th}}$ coordinate, so we may assume that $f$ is homogenous in $x_n$. Since $x_n$ is a unit in $\kk[x_1^{\pm}, \ldots, x_{n-1}^{\pm}, x_n^{\pm}]$, we may furthermore assume that $x_n$ does not occur in $f$ at all.

Let $I \subset \kk[\Lambda^{\vee}]$ be the ideal defining $C$. Now, since $f$ is homogenous in the $w$-grading, the condition that $f$ vanishes on the left hand scheme is that there be a $g \in I$ with $\In_w g=f$. If such a $g$ exists, it is of the form $f+x_n f_1 + x_n^2 f_2 + \cdots + x_n^d f_d$. In particular, $g$ is in the coordinate ring of $Y$ and hence vanishes on the closure of $C$ in $Y$, which is $\phi'(C')$. But then, as $g \equiv f \mod x_n$, we have that $f$ vanishes on $\phi'(C') \cap \{ x_n =0 \}$. This chain of reasoning can easily be reversed. So the same polynomials vanish on the schemes on each side of (\ref{SchemeInitEqual}) and so (\ref{SchemeInitEqual}) is established.

Now, by definition, $w$ is in $\Trop C$ if and only if the left side of (\ref{SchemeInitEqual}) is nonempty. So $w$ is in $\Trop C$ if and only if there is a point $x$ of $\tilde{C}$ on which $\phi'$ is defined and such that $\phi'(x)$ is in $\{ x_n=0 \}$. Now, $\phi'$ is defined at a point $x$ of $\tilde{C}$ if and only if the rational functions $x_i \circ \phi$ on $\tilde{C}$ have no zeroes or poles at $x$ for $i=1$, \dots, $n-1$ and $x_n \circ \phi$ has no pole at $x$. Moreover, $\phi'(x)$ lands in $\{ x_n=0 \}$ if and only if $x_n \circ \phi$ is $0$ at $x$. So we see that $w=e_n$ is in $\Trop C$, if and only if the left hand side of (\ref{SchemeInitEqual}) is nonempty, if and only if the right hand side of (\ref{SchemeInitEqual}) is nonempty, if and only if there is a point $x$ in $\tilde{C}$ such that $x_i \circ \phi$ has neither a zero nor pole at $x$ for $i=1$ \dots, $n-1$ and $x_n \circ \phi$ has a zero there. But this last is precisely the definition of a multiple of $w$ being in the degree of $C$. Moreover, if $x$ is as above, then the length of the scheme $\{ (\phi')^*(x_n) =0 \}$ at $x$ is the order of vanishing of $x_n \circ \phi$ at $x$; this yields the equality of multiplicities. 

It follows that $\sum \sigma_i=0$, as we have now shown that the $\sigma_i$ are the degree of a curve.
\end{proof}

We have just shown that, for a curve $C$ with constant coefficients the unbounded rays of $\Trop C$ correspond to the degree of $C$. In fact this is true for curves that do not have constant coefficients; we leave the details to the reader. 
\begin{prop} \label{TropGivesDeg}
 Let $C \subset \TT(\KK, \Lambda)$ be an algebraic curve. Let $\{ \sigma_1, \ldots, \sigma_N \}$ be the degree of $C$; write $\sigma_i$ as $d_i \rho_i$ where $d_i$ is a positive integer and $\rho_i$ a primitive lattice vector of $\Lambda$. Then the unbounded rays of $\Trop C$ are precisely in the directions $\rho_1$, \dots, $\rho_N$ and the sum of the multiplicities of the rays in direction $\rho_i$ is $d_i$. 
\end{prop}
So, if $C$ is a curve, then $\Trop C$ is a graph of the correct degree. We now finish proving that, if $C$ is a curve, then $\Trop C$ is a zero tension curve.

\begin{prop} \label{TropIsZeroTension}
Let $X \subset \TT(\KK, \Lambda)$ be pure of dimension $d$. Fix any good subdivision $\Sigma$ of $\Trop X$ and define weights $m_{\sigma}$ by the length of $Y_{\sigma}$ as described above. Then $\Sigma$ is a zero tension complex.
\end{prop}

\begin{proof}
Let $\rho$ be a face of $\Sigma$ of dimension $d-1$. Let $\Star_{\rho}$ denote the star of $\Trop X$ at $\rho$, then $\Star_{\rho}$ is invariant under translation by $H(\rho)$ and $\Star_{\rho}/H(\rho) = \Trop Y_{\rho}$. One can easily check that the multiplicities of rays of $\Trop Y_{\rho}$ are the same as the multiplicities of the corresponding faces of $\Trop X$. Thus, we are reduced to checking the claim for $Y_{\rho}$, which is purely one-dimensional and has constant coefficients.

We may thus switch notations, calling what used to be called $Y_{\rho}$ by the name $X$, and assume that $X$ is purely one-dimensional and has constant coefficients. If $X$ is reduced and irreducible, then we have already proven the result in the previous proposition. Unfortunately, the possibility that $X$ is not reduced leads to significant difficulties, to which we now turn.

Let $X_1$, \dots $X_c$ be the irreducible components of $X$ and let $m_i$ be the multiplicity of $X_i$. Let $Z_i$ be the reduction of $X_i$. Clearly, $\Trop X=\bigcup \Trop X_i=\bigcup \Trop Z_i$, since description (1) of $\Trop X$ in Theorem~\ref{DefnsOfTrop} only sees the structure of $X$ as a point set. We claim that, for $w \in \QQ \otimes \Lambda$ spanning a ray of $\Trop X$, the multiplicity of $w$ in $\Trop X$ is the sum, over the components $Z_i$ of $X$, of $m_i$ times the multiplicity of $w$ in $\Trop Z_i$.

As in the proof of  Proposition \ref{ZeroTensionForCCCurves}, we may assume that $w=e_n$. Set $Y=(\kk^*)^{n-1} \times \kk$. Let $X'$ be the closure of $X$ in $Y$ and let $Z'_i$ be the closure of $Z_i$. Then the multiplicity of $w$ in $\Trop X$ is the length of the scheme theoretic intersection $X' \cap \{ x_n =0 \}$ and the multiplicity of $w$ in $\Trop Z'_i$ is the length of $Z'_i \cap \{ x_n=0 \}$. So we want to show that
$$\ell\left(  X' \cap \{ x_n =0 \} \right) =\sum m_i \ell\left(  Z'_i \cap \{ x_n =0 \} \right).$$

We project $Y$ onto the $n^{\textrm{th}}$ coordinate, giving a map $Y \to \kk$. Since $X'$ and $Z'_i$ are defined as the closures of varieties in $(\kk^*)^n$, they do not have any associated primes over the point $0$ of $\kk$. So $X'$ and $Z'_i$ are flat over the point $0$ of $\kk$ and thus, for a generic $\alpha \in \kk^*$, we have  $\ell\left(  X' \cap \{ x_n =0 \} \right)=\ell\left(  X' \cap \{ x_n =\alpha \} \right)=\ell \left(  X \cap \{ x_n =\alpha \} \right)$ and the same applies to each $Z'_i$. So we are reduced to proving
$$\ell\left(  X \cap \{ x_n =\alpha \} \right) =\sum m_i \ell\left(  Z_i \cap \{ x_n =\alpha \} \right)$$
for a generic $\alpha \in \kk^*$. 

For a generic $\alpha$, $\{ x_n= \alpha \}$ will not contain any of the intersections of the $Z_i$, nor any embedded points of the $Z_i$, nor any components $Z_i$ for which $x_n=\textrm{constant}$. For such an $\alpha$, $X \cap \{ x_n =\alpha \}$ is the disjoint union of the schemes $X_i \cap \{ x_n =\alpha \}$ and the length of $X_i \cap \{ x_n =\alpha \}$ is $m_i$ times the length of $Z_i \cap \{ x_n =\alpha \}$.

So, we have now shown that the rays of $\Trop X$ are precisely the rays that occur in $\bigcup \Trop Z_i$ and the multiplicity of a ray in $\Trop X$ is the sum (on $i$) of $m_i$ times that ray's multiplicity in $\Trop Z_i$. Since each $Z_i$ is a reduced curve, we know from Proposition~\ref{ZeroTensionForCCCurves} that the rays of each $\Trop Z_i$, counted with multiplicity, sum to zero. So the same holds for $\Trop X_i$.
\end{proof}

We now know that, if $C$ is a curve, then $\Trop C$ is a zero tension curve of the correct degree. What remains is to show that the first betti number of $\Trop C$ is bounded by the genus of $C$. Unfortunately, if we take the most obvious zero tension curve $(\iota, \Gamma, m)$ with $\iota(\Gamma)=\Trop C$, this may not be true. Consider the following example, which we discussed earlier as Example~\ref{Unfolded}: take $\KK$ to be a power series field over $\kk$. For simplicity, we take $\kk$ not to have characteristic $2$ or $3$. We map $\PP_{\KK}^1$ to $(\KK^*)^3$ by the rational map
$$\phi' : u \mapsto \left( \frac{(u+1)^2}{u(u+t^{-1})},  \frac{(u+t^{-1})^2}{(u+1)},  \frac{(u-1)(u-2 t^{-1})}{(u+1)(u+t^{-1})} \right).$$  

Let $C$ be the open subset of $\PP^1$ on which $\phi'$ is defined. As we computed in Example~\ref{Unfolded}, $\Trop \phi'(C)$ is as shown in Figure \ref{Genus0WithLoop}. Let us check directly that the point $w:=(0,-2,0)$ is in $\Trop \phi'(C)$. Suppose that $u=at+\cdots$, where $a \in \kk \setminus \{ -1,0,1 \}$ and the ellipsis represents an element of valuation higher than the displayed leading term. Then $\phi(u)=(1/a+\cdots, t^{-2}+\cdots,2+\cdots)$. Similarly, if $u=b t^{-2}+\cdots$ with $b \in \kk^*$, then $\phi(u)=(1+\cdots,b t^{-2}+\cdots,1+\cdots)$. So the points of $C$ that are of the form $\phi(at+\cdots)$ or $\phi(b t^{-2}+\cdots)$ have valuation $(0,-2,0)$. All points of $C$ with valuation $(0,-2,0)$ are of one of these forms, as can be verified using the techniques of Section~\ref{GenusZeroProof}.  This shows that, as a point set, $\In_w \phi'(C)=( \kk^* \times \{ 1 \} \times \{ 2 \}) \sqcup ( \{ 1 \} \times \kk^* \times \{1 \})$. In fact, this is an equality of schemes. 

Note, in particular, that $\In_w C$ is disconnected. Intuitively, it seems that $\Trop C$ passes through itself at $w$ and should really be thought of as a tree which just happens to pass by itself at that point. It turns out that the two components of $\In_w C$ correspond to the two branches of $\Trop C$ passing through $w$. We can modify this example by making the third coordinate $(u-1)(u-t^{-1})/(u+1)(u+t^{-1})$, giving the map $\phi$ also discussed in Example~\ref{Unfolded}. Then $\Trop \phi(C)$ is the same as before, but $\In_w \phi(C)$ is now two crossed lines with an embedded point at the crossing. (In other words, the ideal of $\In_w C$ is $\langle (x-1)(y-1), (x-1)(z-1), (y-1)(z-1), (z-1)^2 \rangle$.) So we see that, while the disconnected nature of $\In_w \phi'(C)$ is related to our problems, it is possible to have this difficulty in a case where the initial scheme is connected.

Despite examples such as the above, it is true that we can always find a zero tension curve  $(\iota, \Gamma, m)$ with $\iota(\Gamma)=\Trop C$ and $b_1(\Gamma) \leq g$. We give a slightly more precise statement.

\begin{Theorem} \label{GenusBound}
Continue the above notations $\phi$ and $C$, in particular continue the assumption that $\In_w C$ is reduced for every vertex $w$ of $\Trop C$. Then there is a zero tension curve $(\iota, \Gamma, m)$ with $\Gamma$ connected and the genus of $\Gamma$ less than or equal to the genus of $C$, such that $\iota(\Gamma)=\Trop \phi(C)$. Moreover, if $e$ is any edge of (a good subdivision of) $\Trop \phi(C)$, we have $m_e=\sum_{f \in \iota^{-1}(e)} m_f$.
\end{Theorem}

In order to prove Theorem~\ref{GenusBound}, we will need two constructions. The first is a construction which, given a polyhedral complex, builds a family over $\Spec \OOO$ whose generic fiber is a toric variety and whose special fiber is a union of toric varieties. The second is a construction which allows us to build a single family over $\OO$ whose general fiber is $X$ and whose special fiber contains every $\In_w X$, as $w$ ranges over $\Trop X$. This is a generalization of a construction of Tevelev in \cite{Tev}.

Let $M=\{ \chi_1, \ldots, \chi_N \}$ be a finite subset of $\Lambda^{\vee}$, with cardinality $N$. Then there is a map $i_M: \TT(\KK, \Lambda) \to \PP^{N-1}$ given by $u \mapsto (\chi_1(u) : \ldots : \chi_N(u))$. Let $P$ be the convex hull of $M$. Define the semigroup $S_P \subset \Lambda^{\vee} \times \ZZ$ by $\{ (\chi,m) : \chi \in m \cdot P \cap \Lambda^{\vee} \}$, then we write $\Toric(P)=\Proj(\KK[S_P])$ where the grading on $\KK[S_P]$ is by the last coordinate. At times we will need to record the dependence of $\Toric(P)$ on the field $\KK$, at which point we will write $\Toric(P,\KK)$. We can also describe $\Toric(P)$ as the toric variety associated to the normal fan of $P$. The following result is well known.

\begin{prop}
The map $i_M$ extends to a map from $\Toric(P)$ to $\PP^{N-1}$. The image of this map is the Zariski closure of $i_M(\TT(\KK, \Lambda))$.
\end{prop}

We will need to have a generalization of this result that works in families over $\OOO$. Let $M= \{ \chi_1, \ldots, \chi_N \}$ be a finite subset of $\Lambda^{\vee}$ and let $v_1$, \dots, $v_N \in \ZZ$. Then there is a map $i_{M,v} : \TT(\KK, \Lambda) \to \PP^{N-1}(\KK)$ by $u \mapsto (t^{v_1} \chi_1(u) : \ldots : t^{v_N} \chi_N(u) )$. We will now define a variety $\T(M,v)$ over $\Spec \OO$ which allows us to extend $i_{M,v}$ in a manner similar to that in which $\Toric(P)$ allows us to extend $i_M$.

Let $P$ denote the convex hull of $M$. Let $\psi: P \to \RR$ be the function defined as follows: $\psi(x)=\sup_{\{ \lambda : \lambda(\chi_i) \leq v_i \}} \lambda(x)$, where the supremum is taken over all affine linear functions $\lambda$ such that $\lambda(\chi_i) \leq v_i$ for $i=1$, \dots, $N$. The function $\psi$ is piecewise linear and its domains of linearity are polytopes which are convex hulls of subsets of $M$. Let $P_1$, \ldots, $P_r$ be the domains of linearity of $\psi$, then $\{ P_1, \ldots, P_r \}$ is called the \textbf{regular subdivision of $P$ induced by $v$}. We denote this subdivision by $\DD(P,v)$. By taking a branched cover of $\OOO$, we may assume that, on each $P_i$, the slope of $\psi$ is in $\Lambda$. When this occurs, we will say that $\psi$ is integral.

Let $\KKK[\Lambda^{\vee} \times \NN]$ be the semigroup ring associated to the semigroup $\Lambda^{\vee} \times \NN$. We write elements of $\KKK[\Lambda^{\vee} \times \NN]$ as $\sum a_{\chi,m} e^{\chi,m}$ where $a_{\chi,m}$ is in $\KKK$ and $\chi$ and $m$ are in $\Lambda^{\vee}$ and $\NN$ respectively. Here $e^{\chi,m}$ is a formal symbol where $e^{\chi_1 + \chi_2, m_1+m_2}=e^{\chi_1,m_1} e^{\chi_2,m_2}$. Let $\sss$ be the subring of $\KKK[\Lambda^{\vee} \times \NN]$ consisting of those sums $\sum a_{\chi,m} e^{\chi,m}$ where $\chi \in m \cdot P$ and $v(a_{\chi,m}) \geq m \psi(\chi/m)$. Clearly, $\sss$ is an $\OOO$-algebra, and $\sss$ inherits a grading by $\mathrm{deg}(e^{\chi,m})=\chi$. Let $\T(M,v)=\Proj \sss$. 

\begin{prop}
Suppose that $\psi$ is integral. We have the following consequences: The variety $\T(M,v)$ is flat and projective over $\Spec(\OOO)$. The fiber of $\T(M,v)$ over $\Spec \KKK$ is $\Toric(P, \KKK)$. The fiber of $\T(M,v)$ over $\Spec \kk$ is a union of the toric varieties $\Toric(P_i, \kk)$, glued according to the overlaps between the $P_i$. There is a map $\T(M,v) \to \PP^{N-1}_{\OOO}$ which extends $i_{M,v}$.
\end{prop}

Here $\PP^{N-1}_{\OOO}$ is $\Proj \OOO[x_1,\ldots,x_N]$. 

\begin{proof}
Clearly, $\sss$ is a torsion free $\OOO$-module, and hence flat over $\OOO$. Also, $\sss$ is finitely generated as an $\OOO$-algebra, with generators of positive degree, so the map is projective. The fiber of $\T(M,v)$ over $\Spec \KKK$ is $\Proj (\sss \otimes_{\OOO} \KKK)$ and $\sss \otimes_{\OOO} \KKK$ consists of those sums $\sum a_{\chi,m} e^{\chi,m}$ where $\chi \in m \cdot P$ and $a_{\chi,m}$ is any element of $\KKK$. This is precisely $\KKK[S_P]$.

Now, let us consider $\sss \otimes_{\OOO} \kk$. As an additive group, we see that $\sss \otimes_{\OOO} \kk$ is a $\kk$ vector space with basis $f^{\chi,m}:=t^{m \psi(\chi/m)} e^{\chi,m} \otimes 1$, where $(\chi,m)$ ranges over pairs with $\chi \in m \cdot P$.\footnote{This is where we need the slopes of $\psi$ to be in $\Lambda$, otherwise $m \psi(\chi/m)$ might not be an integer.} The multiplicative structure of $\sss \otimes_{\OOO} \kk$ is that $f^{\chi_1,m_1} f^{\chi_2,m_2}=f^{\chi_1 + \chi_2, m_1+m_2}$ if $(m_1+m_2) \psi((\chi_1+\chi_2)/(m_1+m_2))=m_1 \psi(\chi_1/m_1) + m_2 \psi(\chi_2/m_2)$ and $f^{\chi_1,m_1} f^{\chi_2 m_2}=0$ otherwise. In other words, $f^{\chi_1,m_1} f^{\chi_2,m_2}=f^{\chi_1 + \chi_2, m_1+m_2}$ if the convex function $\psi$ is linear on the line segment connecting $\chi_1/m_1$ to $\chi_2/m_2$. So $\sss \otimes_{\OOO} \kk$ is isomorphic to the subring of $\bigoplus \kk[S_{P_i}]$ where, if $(\chi,m)$ is in both $S_{P_{i_1}}$ and $S_{P_{i_2}}$ then we require that $e^{\chi,m}$ have the same coefficient in $\kk[S_{P_{i_1}}]$ and in $\kk[S_{P_{i_2}}]$. Now, $\Proj \left( \bigoplus \kk[S_{P_i}] \right)$ is $\bigsqcup \Toric(P_i, \kk)$ and imposing the stated conditions on the coefficients of the $e^{\chi,m}$ precisely glues these toric varieties together.

Finally, we want to describe the map $\T(M,v) \to \PP^{N-1}_{\OOO}$. We know that $\PP^{N-1}_{\OOO}=\Proj \OOO[x_1,\ldots,x_N]$. We map $\OOO[x_1,\ldots,x_N]$ to $\sss$ by sending $x_i$ to $t^{v_i} e^{\chi_i,1}$. Let $I$ be the ideal of $\sss$ generated by $t^{v_i} e^{\chi_i,1}$ as $i$ runs from $1$ to $N$; this is the preimage of the irrelevant ideal of $ \OOO[x_1,\ldots,x_N]$. In order to show that this gives a map $\T(M,v) \to \PP^{n-1}_{\OOO}$ we just have to check that the radical of $I$ is the irrelevant ideal of $\sss$.  Clearly, $I$ is homogenous for the grading of $\sss$ by $\Lambda^{\vee} \times \ZZ$, so it is enough to check that, for any monomial $c e^{\chi,m}$ which is in the irrelevant ideal of $\sss$, some power of $c e^{\chi,m}$ is in $I$. The condition that $c e^{\chi,m}$ be in the irrelevant ideal of $\sss$ is simply that $m$ is positive.

So, consider some monomial $c e^{\chi,m}$ of $\sss$ with $m >0$. Then $\chi/m \in P$ and $v(c) \geq m \psi(\chi/m)$. Let $P_i$ be the domain of linearity for $\psi$ in which $\chi/m$ falls. Let $\chi_{i_1}$, \dots, $\chi_{i_k}$ be the vertices of $P_i$. Then $v_{i_j}=\psi(\chi_{i_j})$ for $j=1$, \dots, $k$. Since $\chi/m$ is in the convex hull of the $v_{i_j}$, we can write $\chi/m$ as $\sum a_j \chi_{i_j}$ where $a_{j}$ are nonnegative rational numbers with $\sum a_j =1$. Since $\psi$ is linear on $P_i$, we have $\psi(\chi/m)=\sum a_j \psi(\chi_{i_j})=\sum a_j v_{i_j}$. Clearing out denominators, we can find nonnegative integers $b_j$ with $\sum b_j=M$, $\sum b_j \chi_{i_j}=M (\chi/m)$ and $\sum b_j v_{i_j}=M \psi(\chi/m)$. We may also assume that $m$ divides $M$. Then 
$$\left( t^{m \psi(\chi/m)} e^{\chi,m} \right)^{(M/m)}=\prod \left( t^{v_{i_j}} e^{\chi_{i_j},1} \right)^{b_j}.$$
So we see that $ t^{m \psi(\chi/m)} e^{\chi,m}$ is in the radical of $I$. But $v(c) \geq m \psi(\chi/m)$ so $c e^{\chi,m}$ is in the $\OOO$-module generated by the radical of $I$ and is, hence, in the radical of $I$ itself.

We now have shown that there is a well defined map $\T(M,v) \to \PP_{\OOO}^{N-1}$. The torus $\TT(\KKK, \Lambda)$ is canonically an open subvariety of $\Toric(P, \KKK)$ and hence an open subvariety of $\T(M,v)$. Explicitly, $\TT(\KKK, \Lambda)$ corresponds to the localization $\KKK[\Lambda^{\vee}]$ of $\sss$. The map $\OOO[x_1,\ldots,x_N] \to \KKK[\Lambda^{\vee}]$ is given by $x_i \to t^{w_i} e^{\chi_i,1}$, which precisely corresponds to $i_{M,v}$.
\end{proof}

\textbf{Remark:} There are two variants of this construction which have been previously published. Firstly, one may consider the limit inside $\PP^{N-1}$ of the image of $i_{M,v}$. As a point set, this limit is a union of toric varieties, but it can have a quite sophisticated nonreduced structure, see \cite{GBTV}. By contrast, after a base extension to ensure that $\psi$ is integral, the fiber of $\T(M,v)$ over $\Spec \kk$ is always reduced, but the map from this fiber to $\PP^{N-1}_{\kk}$ may fail to be a closed immersion. Secondly, just as toric varieties can be described by fans rather than by polytopes, there is a description of this construction in terms of fans in \cite{SiebNish}. Our approach may be thought of as an intermediate ground between these approaches, which keeps the map to projective space available but retains the reduced varieties of the more abstract version.

\begin{prop} \label{twLandsWhere}
Let $w \in \QQ \otimes \Lambda$. Extend $\OOO$ if necessary so that $t^{-w}$ is a point of $\TT(\KKK, \Lambda)$ and $\psi$ is integral. Let $F$ be the face of the regular decomposition $\DD(P,v)$ of $P$ on which $\psi(x)- \langle w,x \rangle$ is minimized. Then the limit in $\T(M,v)$ of $t^{-w}$ is in the interior of the stratum $\Toric(F, \kk)$ of $\T(M,v) \times_{\Spec \OOO} \Spec \kk$. 
\end{prop}

\begin{proof}
Let $\mu$ be the minimum value of $\psi(x)- \langle w,x \rangle$ as $x$ ranges through $P$. The embedding of the point $t^{-w}$ into $\T(M,v)$ corresponds to the map of graded rings $\sss \to \KKK[z]$ by $c e^{\chi,n} \mapsto c t^{-\langle w,n \chi \rangle} z^n$. If we rescale the $n$-th graded component of $\KK[z]$ by $\alpha^n$ for some $\alpha \in \KKK^*$, we get the same map $\Proj \KKK[z]=\Spec \KKK \to \T(M,z)$. We rescale by $t^{\mu}$, so that our map is now $c e^{\chi,n} \mapsto c t^{-\langle w,n \chi \rangle - n \cdot \mu} z^n$. Since $c e^{\chi,n} \in \sss$ implies that $v(c) \geq n \psi(\chi) \geq n(  \langle w, \chi \rangle + \mu)$, this map extends to a map $\phi: \sss \to \OOO[z]$. We claim this gives us a map $\Proj \OOO[z]=\Spec \OOO \to \T(M,v)$. Once we know this, the image of $\Spec \kk$ will be the limit of $t^{-w}$. 

Our fear is that the preimage of $0$ under $\sss \otimes_{\OOO} \kk \to \kk[z]$ might be the irrelevant ideal. Now, $\sss \otimes_{\OOO} \kk$ is spanned by the images of $t^{n \psi(\chi/n)} e^{\chi,n}$. We write $f^{\chi,n}$ for the image in $\sss \otimes_{\OOO} \kk$ of $t^{n \psi(\chi/n)} e^{\chi,n}$. We see that $f^{\chi,n}$ maps to $0$ in $\kk[z]$ if and only if $n \psi(\chi/n) -\langle w,n \chi \rangle - n \cdot \mu > 0$, in other words, if and only if $\chi/n \not \in F$. So we see that the preimage of $0$ is not the irrelevant ideal, as it does not contain $f^{\chi,n}$ for $\chi/n \in F$.

Moreover, we see that the map $\Spec \kk \to \T(M,v) \otimes_{\OOO} \kk $ factors through the subvariety $\Toric(F,\kk)$ 
but not through $\Toric(F',\kk)$ for any proper face $F'$ of $F$, so we get that $\Spec \kk$ lands in the interior of $\Toric(F,\kk)$.
\end{proof}

Now, let $X$ be a subvariety of $\TT(\KKK, \Lambda)$. Choose an isomorphism of $\Lambda$ with $\ZZ^{n}$, and hence an embedding of $\TT(\KKK, \Lambda)$ into $\PP^n$. We will therefore feel free to write $(\KKK^*)^n$ in place of $\TT(\KKK, \Lambda)$. Let $\overline{X}$ be the closure of $X$ in $\PP^n$. Then we have an embedding of $(\KKK^*)^n$ into the Hilbert scheme $\Hilb(\PP^n, \KKK)$ by $u \mapsto [ u^{-1} \cdot \overline{X}]$. Let $T$ be the closure of $(\KKK^*)^n$ in $\Hilb(\PP^n,\OOO)$ under that map.\footnote{Here $\Hilb(\PP^n, \OOO)$ is the scheme over $\OOO$ equipped with a universal closed subscheme $U$ of $\Hilb(\PP^n, \OOO) \times_{\OOO} \PP^n_{\OOO}$ so that, for every scheme $S$ with a map $S \to \Spec \OOO$, every closed subscheme of $S \otimes_{\Spec \OOO} \PP^n_{\OOO}$ which is flat over $S$ is  pulled back in a unique manner from $U$ along a map $S \to \Hilb(\PP^n, \OOO)$.}

\begin{prop} \label{HilbProj}
There is a projective embedding of $\iota : T \into \PP_{\OOO}^{N-1}$, for some $N$, and a set $M=\{ \chi_1, \ldots, \chi_N \} \subset \Lambda^{\vee}$ and some scalars $v_1$, \ldots, $v_N \in \ZZ$ such that the restriction to $\TT(\KK, \Lambda)$ of $\iota$ is $i_{M,v}$.
\end{prop}

\begin{proof}
This proof amounts to unraveling the proof that $\Hilb(\PP^n)$ is projective. Let $h$ denote the Hilbert polynomial of $\overline{X}$. By a result of Gotzmann \cite{Gotz}, there is a positive integer $d$ such that, if $Y$ is any subscheme of $\PP_{\OOO}^n$ which is flat over $\OOO$ with hilbert polynomial $h$, then $Y$ is determined by the degree $d$ terms of its defining ideal $I(Y)$ and $I(Y)_d$ has codimension $h(d)$ for every such $Y$. Then we get a projective embedding of the component of $\Hilb(\PP^n, \OOO)$ containing $\overline{X}$ by first mapping $Y$ to the point $I(Y)_d$ of the Grassmannian $\Gr(h(d), \Sym^d(\OOO^{n+1}))$ and then embedding this Grassmannian into projective space by the Pl\"ucker embedding. We need to be a bit more specific: we take coordinates on $\Sym^d(\OOO^{n+1})$ to be the obvious coordinates $e_{i_1 \cdots i_d} := e_{i_1} \cdots e_{i_d}$, with $1 \leq i_1 \leq \ldots \leq i_d \leq n$ and use the obvious Pl\"ucker coordinates $p_{(i_1^1 \cdots i_d^1 ) \cdots (i_1^{h(d)} \cdots i_d^{h(d)})} := e_{i_1^1 \cdots i_d^1} \wedge \cdots \wedge e_{i_1^{h(d)} \cdots i_d^{h(d)}}$ on the Grassmannian. Let $S$ be the set of indices $((i_1^1,\ldots, i_d^1), \ldots, (i_1^{h(d)}, \ldots, i_d^{h(d)}))$ for which $p_{(i_1^1 \cdots i_d^1 ) \cdots (i_1^{h(d)} \cdots i_d^{h(d)})} (\overline{X}) \neq 0$. Every Pl\"ucker coordinate not indexed by $S$ is $0$ on $T$, so $T$ is embedded into $\PP^{|S|-1}$ by the Pl\"ucker coordinates $p_s$, $s \in S$. We call this map $\tilde{\iota}$.

Now, it is easy to check that the action of $(\KKK^*)^n$ on $\Hilb(\PP^{n}, \KKK)$ rescales the Pl\"ucker coordinates $p_s$ by a character of $(\KKK^*)^n$. So the map $\tilde{\iota}$, restricted to $(\KKK^*)^n$ is of the form $u \mapsto ( a_s \chi_s(u))_{s \in S}$ for some characters $\chi_s$ of $(\KKK^*)^n$ and some elements $a_s \in \KKK^*$. Rescaling our target by an element of $(\OOO^*)^S$, we can assume that $a_s$ is of the form $t^{v_s}$ for some $v_s \in \ZZ$. This is the required claim.
\end{proof}

We continue to let $X$ be a subvariety of $(\KKK^*)^n$, which we view as embedded in $\PP^n_{\KKK}$. Let $(M,v)$ be as in the preceding proposition, let $P$ be the convex hull of $M$, let $\psi$ be as before  and let $\DD(P,v)$ be the regular subdivision of $P$ obtained from $\psi$. Extend $\OOO$ if necessary so that $\psi$ is integral. Let $\XX$ be the closure of $X$ in $\T(M,v)$. We'll write $\XX_0$ for the fiber of $\XX$ over $\Spec \kk$ and $\XX_1$ for the fiber of $\XX$ over $\Spec \KKK$. Clearly, $\XX_1$ is a compactification of $X$ over $\KKK$ and $\XX_0$ is a flat degeneration of $\XX_1$. 

\begin{prop} \label{DegenerationLemma}
Let $F$ be a face of $\DD$. Let $w \in \QQ \otimes \Lambda$ be such that the function $x \mapsto \psi(x)-\langle w,x \rangle$ is minimized precisely on $F$; extend $\OOO$ if necessary so that $t^{w} \in \TT(\KKK, \Lambda)$. Let $H \subset \QQ \otimes \Lambda^{\vee}$ be the linear space spanned by $f-f'$, for $f$ and $f' \in F$. Let $H^{\perp} \subset \QQ \otimes \Lambda$ be the orthogonal space to $H$ and let $\exp(H^{\perp}) \subset \TT(\kk, \Lambda)$ be the connected subtorus of $ \TT(\kk, \Lambda)$ associated to $H^{\perp}$. 

Then $\In_w X$ is invariant under $\exp(H^{\perp})$ and the intersection of $\XX_0$ with the interior of $\Toric(F,\kk)$ is $\In_w X/\exp(H^{\perp})$.
\end{prop}

This proof is closely modeled on an argument of Tevelev, see \cite{Tev}. 

\begin{proof}

By our construction of $(M,v)$, there is a map $\T(M,v) \to \Hilb(\PP^n, \OOO)$. Let $\tilde{\EE}$ be the pull back to $\T(M,v)$ of the universal family over $\Hilb(\PP^n, \OOO)$ and let $\EE$ be the intersection of $\tilde{\EE}$ with the open torus in $\PP^n$.  We will write $\Pi$ for the torus bundle over $\T(M,v)$ in which $\EE$ lives. Let $Z$ be the section of $\Pi \to \T(M,v)$ given by the point $(1:\ldots:1)$ in each fiber. (Note that this fiber is sometimes $(\KKK^*)^n$ and sometimes $(\kk^*)^n$.) As in \cite{Tev}, we see that the closure of $X$ in $\T(M,v)$ is the projection to $\T(M,v)$ of $Z \cap \EE$.  

Consider $t^w$ as a point in $(\KK^*)^{n}$. Let $z_w$ be the limit of $t^w$ over the zero fiber of $\T(M,v)$. (\emph{I.e.} take the Zariski closure of the point $t^w$ in $\T(M,v)$ and intersect it with the zero fiber.) The fiber of $\EE$ over $z_w$ is $\In_w X$. By Proposition~\ref{twLandsWhere}, $z_w$ is in the interior of $\Toric(F)$. 

Now, let us consider what $\Pi$, $\EE$ and $Z$ look like over the interior of $\Toric(F)$. All statements in this paragraph are implicitly limited to the parts of our families that live over the interior of $\Toric(F)$. The interior of $\Toric(F)$ is the torus $\TT(\kk, \Lambda)/\exp(H^{\perp})$. The torus bundle $\Pi$ is isomorphic to $\left( \TT(\kk, \Lambda)/\exp(H^{\perp}) \right) \times \TT(\kk, \Lambda) $. The torus $\TT(\kk, \Lambda)$ acts on this product by multiplication by $(u,u^{-1})$, and $\EE$ is taken to itself under this action. Thus, in particular, each fiber of $\EE$ is preserved by $\exp(H^{\perp})$. We have seen in the preceding paragraph that the fiber of $\EE$ over $z_w$ is $\In_w X$, so we see that $\In_w X$ is preserved by $\exp(H^{\perp})$. Moreover, since this torus action takes $\EE$ to itself, we see that, identifying $\Pi$ with $\left( \TT(\kk, \Lambda)/\exp(H^{\perp}) \right) \times \TT(\kk, \Lambda) $, the part of $\EE$ over the interior of $\Toric(F)$ is $\{ ([u], u^{-1} \cdot x): u \in \TT(\kk, \Lambda), x \in \In_w X \}$ where $[u]$ is the class of $u$ in $\TT(\kk, \Lambda)/\exp(H^{\perp})$.  So $\EE \cap Z$ is $\{ ([x], (1:\ldots:1)) : x \in \In_w X \}$ and the projection of $\EE \cap Z$ onto the interior of $\Toric(F)$ is $\In_w X/\exp(H^{\perp})$ as desired.
\end{proof}

\begin{cor} \label{GoodSubDiv}
Let $\tilde{\Sigma}$ be the polyhedral subdivision of $\QQ \otimes \Lambda$ for which $w_1$ and $w_2$ are in the relative interior of the same face if $\psi(x)-\langle w_1,x \rangle$ and $\psi(x)-\langle w_2,x \rangle$ are minimized on the same face of $\DD$. Then $\Trop X$ is supported on a subcomplex of $\tilde{\Sigma}$ and the induced subdivision of $\Trop X$ is a good subdivision.
\end{cor}

We will write $\Sigma$ for the subcomplex of $\tilde{\Sigma}$ containing $\Trop X$.

\begin{proof}
We showed that $\In_w X$ depends only on the face of $\tilde{\Sigma}$ in whose relative interior $w$ lies. Now, $w \in \Trop X$ if and only if $\In_w X$ is empty. So $\Trop X$ is supported on a subcomplex of $\tilde{\Sigma}$. The fact that $\In_w X$ is constant on the relative interiors of the faces of $\tilde{\Sigma}$ is the definition of a good subdivision.
\end{proof}

We may hence use our notation $Y_{\sigma}$, where $\sigma$ is a face of $\Sigma$. We summarize our accomplishments:

\begin{prop}
Let $X$ be a subvariety of $\TT(\KKK, \Lambda)$. Then, after possibly taking a branched cover of $\OOO$, there is a compactification $\XX_1$ of $X$ (over $\KKK$), a flat projective degeneration (over $\OOO$) of $\XX_1$ to a variety $\XX_0$ (over $\kk$) and a good subdivision $\Sigma$ of $\Trop X$ such that $\XX_0=\bigcup_{\sigma \in \Sigma} Y_{\sigma}$, with the $Y_{\sigma}$ disjoint inside $\XX_0$ and $Y_{\sigma_1}$ in the closure of $Y_{\sigma_2}$ only if $\sigma_2 \subseteq \sigma_1$. 
\end{prop}

We are now closing in on the proof of Theorem~\ref{GenusBound}, in the case where $\In_w C$ is reduced for every vertex $w$ of $\Trop C$. Let $C_0$ be the degeneration of $C$ described above. The idea is that, since $C_0$ is a degeneration of a genus $g$ curve, its components can't form a graph of first betti number more than $g$. We have built $C_0$ so that the structure of $\Trop C$ more or less captures the structure of the components of $C_0$. There are two main things that can go wrong. First, if $\sigma$ is a vertex of $\Trop C$, it is possible that $Y_{\sigma}$ is disconnected and should thus correspond to multiple vertices of $\Gamma$. We will deal with this by defining our graph $\Gamma$ appropriately. Secondly, if $\tau$ is an edge of $\Trop C$, we might worry that the local structure of $C_0$ near $Y_{\tau}$ looks like two crossing lines with some extra non-reduced structure at the crossing point, in which case it shouldn't correspond to an edge of $\Gamma$. To deal with the latter obstacle, we need the following lemma. The proof was suggested to me by Paul Hacking.
 
 \begin{lemma} \label{ExactComplex}
 Let $X$, $\Sigma$, $\XX_1$ and $\XX_0$ continue to have the above meanings. For $\sigma$ a face of  $\Sigma$, let $Z_{\sigma}$ be the scheme theoretic closure of $Y_{\sigma}$. Then the complex of sheaves
 $$0 \to \OO_{\XX_0} \to \bigoplus_{\dim \sigma=0} \OO_{Z_{\sigma}} \to \bigoplus_{\dim \sigma=1} \OO_{Z_{\sigma}} \to \cdots \to \bigoplus_{\dim \sigma=\dim X} \OO_{Z_{\sigma}}\to 0$$
 is exact. Here the maps are the obvious restriction maps from $\OO_{Z_{\sigma}}$ to $\OO_{Z_{\tau}}$ whenever $\sigma$ is a facet of $\tau$, sign twisted in the usual fashion.
 \end{lemma}
 
 \begin{proof}
 Recall our notation $\T(M,v)$ for the toric degeneration in which $\XX_0$ sits. Let $T_0$ be the fiber of $\T(M,v)$ over $\Spec \kk$ and let $T(F)$ be the closed face of $T_0$ associated to the face $F$ of $\Sigma$. Then the complex of sheaves
 $$0 \to \OO_{T_0} \to \bigoplus_{\dim \sigma=0} \OO_{T(\sigma)} \to \bigoplus_{\dim \sigma=1} \OO_{T(\sigma)} \to \cdots \to \bigoplus_{\dim \sigma=n} \OO_{T(\sigma)} \to 0$$
 is easily exact. Recall the notations $\Pi$, $\EE$ and $Z$ from the proof of Proposition~\ref{DegenerationLemma}., and write $p$ for the projection $\EE \to \T(M,v)$. Let $\EE(\sigma)$ denote the portion of $\EE$ lying over $T(\sigma)$. Then 
 $\EE$ is flat over $T_0$ (it is pulled back from the universal family over the Hilbert scheme) so the complex
 $$0 \to \OO_{\EE(T_0)} \to \bigoplus_{\dim \sigma=0} \OO_{\EE(\sigma)} \to \bigoplus_{\dim \sigma=1} \OO_{\EE(\sigma)} \to \cdots \to \bigoplus_{\dim \sigma=n} \OO_{\EE(\sigma)} \to 0$$
 is exact. But $\EE$ is transverse to $Z$ inside $\Pi$, so the corresponding complex with $\EE(\sigma)$ replaced by $\EE(\sigma) \cap Z$ is exact. As shown in the proof of Proposition~\ref{DegenerationLemma}, $\EE \cap Z$ projects down isomorphically to $\XX_0$. We claim that, similarly, $\EE(\sigma) \cap Z$ projects down isomorphically to $Z_{\sigma}$, so the complexes $\bigoplus_{\dim \sigma=\bullet} \OO_{Z_{\sigma}}$ and $\bigoplus_{\dim \sigma=\bullet} \OO_{\EE(\sigma) \cap T(\sigma)}$ are the same. Since $p(\EE(\sigma) \cap T(\sigma)) \supseteq Y_{\sigma}$ and $p(\EE(\sigma) \cap T(\sigma))$ is closed, we know that $p(\EE(\sigma) \cap Z) \supseteq Z_{\sigma}$. We want to show that if, on some open set $U$, a function $f$ vanishes on $U \cap Y_{\sigma}$ then it vanishes on $p(\EE(\sigma) \cap Z) \cap U$. 

Write $\EE^{\circ}(\sigma)$ for the part of $\EE$ over the interior of $T(\sigma)$. Since $\EE$ is flat over $T(\sigma)$, we know that, as schemes, $\EE(\sigma)$ is the closure of $\EE^{\circ}(\sigma)$. By the transversality result, $\EE(\sigma) \cap Z$ is the closure of $\EE^{\circ}(\sigma) \cap Z$. With $f$ and $U$ as in the previous paragraph, since $p: \EE^{\circ}(\sigma) \cap Z \to Y_{\sigma}$ is an isomorphism, we know that $p^* f$ vanishes on $\EE^{\circ}(\sigma) \cap Z \cap p^{-1}(U)$. Then $p^* f$ also vanishes on $\EE(\sigma) \cap Z \cap p^{-1}(U)$ and hence, using that $p$ is an isomorphism once more, $f$ vanishes on $p(\EE(\sigma) \cap Z) \cap U$. 
 \end{proof}

We also need the following lemma:

\begin{prop} \label{ReducedClosure}
In the above notation, assume that $Y_{\sigma}$ is reduced. Then $\overline{Y_{\sigma}}$ is also reduced.
\end{prop}

 \begin{proof}
 In general, if $U$ is an open subscheme of some variety $V$, and $W$ is a closed subscheme of $U$, then $\overline{W}$ is reduced. This is a local fact, so we can check the corresponding affine property: Let $A$ be a ring, $S$ a multiplicative subset of $A$, and $D$ a domain with $\pi:S^{-1} A \to D$ a surjection. Then we are to show that the pullback of $\ker pi$ to $A$ is a prime ideal or, in other words, the image of $A$ under $A \to S^{-1} A \to D$ is a domain.  A subring of a domain is always a domain.
 \end{proof}
   
 Now, finally, we can prove Theorem~\ref{GenusBound}.

 \begin{proof}[Proof of Theorem~\ref{GenusBound}]
 Let $C$ be a curve in $\T(\KK, \Lambda)$ and let, $\Sigma$, $C_1$ and $C_0$ be as above. Then $\Trop C$ is one dimensional. $\Trop C$ inherits the structure of a one dimensional polyhedral complex, that is to say, a graph, from $\Sigma$.  Define a graph $\Gamma$ as follows: for each vertex $\sigma$ of $\Trop C$, let $Y_{\sigma}^1$, \dots, $Y_{\sigma}^{r_{\sigma}}$ be the connected components of the closure of $Y_{\sigma}$. There will be a vertex $[Y_{\sigma}^i]$ of $\Gamma$ for each $Y_{\sigma}^i$, as $\sigma$ ranges over all the vertices of $\Trop C$. 
 
 Now, consider any edge $\tau$ of $\Trop C$. Let $Y_{\tau}^1$, \ldots, $Y_{\tau}^{r_{\tau}}$ be the connected components of $Y_{\tau}$. If $Y_{\tau}^j \subset Y_{\sigma}^{i}$ for some $(\sigma, i)$ then we must have $\sigma \subset \tau$, \emph{i.e.}, $\sigma$ must be one of the two end points of $\tau$. Moreover, if $\sigma$ is one of the endpoints of $\tau$, then there is clearly at most one $i$ such that $Y_{\tau}^j \subset Y_{\sigma}^i$ because we defined $Y_{\sigma}^i$ to be the connected components of $\overline{Y_{\sigma}^i}$. In fact, we claim that there is such a $Y_{\sigma}^i$. Let $u$ be the direction of the edge $\tau$ pointing away from $\sigma$, then $Y_{\tau} \times \kk^* \isomorph \In_u Y_{\sigma}$. So every point of $Y_{\tau}$ is the limit of some path in $Y_{\sigma}$. There is an edge of $\Gamma$ for each connected component $Y_{\tau}^j$ of a $Y_{\tau}$, as $\tau$ ranges over the edges of $\Trop X$. This edge is denoted $[Y_{\tau}^j]$. The endpoints of the edge $[Y_{\tau}^j]$ are $[Y_{\sigma_1}^{i_1}]$ and $[Y_{\sigma_2}^{i_2}]$ where $\sigma_1$ and $\sigma_2$ are the endpoints of $\tau$ and $Y_{\sigma_r}^{i_r}$ is the connected component of $\overline{Y_{\sigma_r}}$ containing $Y_{\tau}^j$.
 
 We now define the weighting on the edges of $\Gamma$. The edge $[Y_{\tau}^j]$ is given weight $\ell(Y_{\tau}^j)$, the length of the zero dimensional scheme $Y_{\tau}^j$. Next, we describe the map $\iota$ from $\Gamma$ to $\Trop C$. We map $[Y_{\sigma}^i]$ to $\sigma$ and we map $[Y_{\tau}^j]$ to $\tau$. 
 
 It is transparent that $\Trop C$ is the image of $\iota$. The condition on weights holds because $\ell(Y_{\tau})=\sum_{j=1}^{r_j} \ell(Y_{\tau}^j)$. The zero tension condition holds at vertices of $\Gamma$ because, for every vertex $[Y_{\sigma}^i]$, $\Trop Y_{\sigma}^i$ is simply given by the rays in the direction of $\tau$ for those edges $\tau$ such that $Y_{\tau}$ is in $Y_{\sigma}^i$, so the zero tension condition for $\Gamma$ holds by the zero tension condition applied to the $Y_{\sigma}^i$. 
 
Finally, (using Proposition~\ref{ReducedClosure}), each $Y_{\sigma}^{i}$ is a reduced, connected, projective curve. So the holomorphic Euler characteristic of each $Y^{i}_{\sigma}$ is at most $1$. On the other hand, each $Y^{j}_{\tau}$ is zero dimensional. So the holomorphic Euler characteristic of each $Y^{j}_{\tau}$ is at least $1$. Now, by Proposition~\ref{ExactComplex}, $\chi(C_0)=\sum \chi(\OO_{Y_{\sigma}^{i}}) - \sum \chi(\OO_{Y_{\tau}^{j}})$. So $\chi(C_0)$ is at most the difference between the number of vertices of $\Gamma$ and the number of edges, \emph{i.e.}, the Euler characteristic of $\Gamma$. But $C_0$ is a flat degeneration of $C$, which is a curve of genus $g$, so $\chi(C_0) =1-g$. We deduce that $\chi(\Gamma) \geq 1-g$ and hence $b_1(\Gamma)$ is at most $g$. If any of the components $Y_{\sigma}$ is not rational, then $b_1(\Gamma)$ is strictly less then $g$.
 \end{proof}

We see that, if any of the components of $C_0$ are not rational, or if any of them meet along singularities other than nodes, we can take $\Gamma$ to have first Betti number less than $g$. A similar situation occurs for the zero tension curve constructed in the proof of Siebert and Nishinou, where $C_0$ is replaced by the stable limit of $C$. We would like to use this observation to prove that the stable limit of $C$ consists only of rational components in certain cases. Unfortunately, it is possible that $C_0$ does not have this form but that there is some other $(\iota', \Gamma',m')$ for which $b_1(\Gamma')=g$. However, when $(\phi,C)$ is trivalent, we can rule this possibility out.

\begin{Corollary} \label{TotRat}
Suppose that $(\phi,C)$ is trivalent and that $(\iota, \Gamma, m)$ is a zero tension curve with $\iota(\Gamma)=\Trop \phi(C)$, with $\iota$ injective and such that the genus of $C$ is equal to $b_1(\Gamma)$. Let $\overline{C}$ be the projective curve (over $\KKK$) containing $C$ and let $\{ x_1, \dots, x_N \} = \overline{C} \setminus C$.

Consider $(\overline{C}, \{x_1, \dots, x_N \})$ as a stable curve over $\KKK$ with $N$ marked points and (taking a branched cover if necessary) extend  $(\overline{C}, \{x_1, \dots, x_N \})$ to a flat, proper family over $\Spec \OOO$, of stable curves with marked points. The zero fiber of this family is a union of genus zero curves.
\end{Corollary}

\begin{proof}
If not, then there is some other zero tension curve $(\iota', \Gamma', m')$ for $\phi(X)$ with $b_1(\Gamma')<g$, as discussed above. But then, by Lemma~\ref{trivalent}, we may make $\iota'$ injective by restricting to a subgraph of $\Gamma'$; this can only lower $b_1(\Gamma')$. So we may assume that $\iota'$, like $\iota$ is injective, and thus $\Gamma \isomorph \Gamma'$.  But $b_1(\Gamma)=g > b_1(\Gamma')$, a contradiction.
\end{proof}

\end{document}